\documentclass[11pt]{amsart}
\theoremstyle{plain}
\newtheorem{thm}{Theorem}[section]
\newtheorem{theorem}[thm]{Theorem}

\newtheorem{lemma}[thm]{Lemma}
\newtheorem{corollary}[thm]{Corollary}
\newtheorem{proposition}[thm]{Proposition}
\theoremstyle{definition}
\newtheorem{remark}[thm]{Remark}

\newtheorem{example}[thm]{Example}

\newtheorem{question}[thm]{Question}

\numberwithin{equation}{section}

 \title[Bounded birationality and isomorphism problems are computable]{Bounded birationality and isomorphism problems are computable}
 \author{Tuyen Trung Truong}
  \address{Department of Mathematics, University of Oslo, Blindern 0851 Oslo, Norway}
  \email{tuyentt@math.uio.no}
    \date{\today}

\begin{document}
\maketitle
\begin{abstract}
Let $X,Y$ be two irreducible subvarieties of the projective space $\mathbb{P}^n$, and $d\geq 1$ an integer number. The main result of this paper is an algorithm to construct {\bf explicitly}, in terms of $d$ and the ideals defining $X$ and $Y$, a quasi-affine algebraic variety parametrising the set of all birational maps $f$ from $X$ onto $Y$ which can be extended to a self-rational map of $\mathbb{P}^n$ of degree $\leq d$.  

Based on this result, we propose an approach towards the rationality problem (see Section 3 below), solve it for some simple cases (varieties of general type or curves), and state a rough strategy for reducing it to some simpler cases via Iitaka's fibrations.  

We also prove similar results for the case $f$ is  a dominant rational map, regular morphism, isomorphism or regular embedding. Similar results are valid for varieties over an arbitrary algebraically closed field, and also for maps on non-projective varieties. 
\end{abstract}

\section{Introduction} The main theme in this paper is to show that the existence of rational maps from a given projective variety $X\subset \mathbb{P}^n$ to another projective variety $Y\subset \mathbb{P}^n$ with certain interesting properties (such as birational, dominant, regular, isomorphic or regular embedding) and with {\bf bounded degree} is computable, that is can be detected by an algorithm whose complexity is explicitly bound. We will treat first the main case of interest, that of birational maps. With some minor modifications, the proofs for the projective varieties also work for non-projective varieties. In particular, (biregular) isomorphisms of bounded degrees between two irreducible, smooth algebraic varieties can be parametrised by an algebraic variety. These results are valid over fields of any characteristic, which is a not-negligible point, given that many results known in characteristic zero is not known in positive characteristic. For example, it is still unknown whether resolution of singularities hold in positive characteristic, or to what extend the known results in the Minimal Model Program can be done in positive characteristic. 

The results in this paper give support to affirmative answers to the following questions, which to our knowledge is unknown in the general case:

\begin{question} a) Let $X$ and $Y$ be irreducible algebraic varieties. Can the set of birational maps between $X$ and $Y$ be parametrised by a scheme? b) Let $X$ and $Y$ be smooth irreducible algebraic varieties (not necessarily projective). Can the set of biregular isomorphisms between $X$ and $Y$ be parametrised by a scheme?  
\end{question} 

We note that in general deciding biregular isomorphisms requires much more work than deciding birational maps. Based on the main results in this paper, we propose a rough strategy towards the rationality problem, via Iitaka's fibrations, in Section 3.  

{\bf Acknowledgments.} We would like to thank  John Christian Ottem and Kristian Ranestad for their generous help and many useful discussions and comments concerning Section 2, and to John Christian Ottem also for discussions involving Section 3. We also would like to thank Massimilliano Mella for answering some questions concerning the paper \cite{mella-polastri}, and thank De-Qi Zhang for answering some questions concerning the paper \cite{birkar-zhang}. Other comments from various people, which helped to improve the presentation of the paper, are also very much appreciated.   

\section{Birational maps of bounded degrees} This section treats the class of maps, reflected in the title of this paper, which is of main interest to us: birational maps.

The birationality problem asks whether there is a birational map $f$ between two given irreducible algebraic varieties $X$ and $Y$. It is a fundamental and classical question in algebraic geometry. Having the seminal G\"odel theorem that many mathematical questions are undecidable, it is at least psychologically important to ask the question: Is the birationality problem decidable? This question is also practically important. A special case of the birationality problem, the so-called rationality problem, which seeks to check whether a given variety is birationally equivalent to a projective space, has attracted a lot of attention and effort. As far as we know, even this special case is still open in general.  To this end, in this paper, we prove that a weaker version of the birationality question, called {\bf bounded} birationality problem, is not only decidable but also even {\bf computable}. That is, there is an algorithm - whose complexity is explicitly bound - to solve the weaker version. As a consequence, the union of all varieties $W(X,Y,d)$ in Theorem \ref{TheoremMain} below, where $d$ runs all over $\mathbb{N}$, is a countable complete set of invariants for the birationality problem. (Note that it is usually the case in mathematics that a property is determined via a countable set of invariants. For example, the Kodaira dimension of a variety $X$ is defined based on the behaviour of the sequence $h^0(X,K_X^{\otimes m})$. Related questions for unirationally connected varieties and uniruled varieties, for a smooth projective variety $X$, are also stated in terms of countable invariants: in the first case it is Mumford's conjecture which concerns the vanishing of all $h^0(X,(T_X^*)^{\otimes m})$, in the second case it is Mori's conjecture which concerns the vanishing of all $h^0(X,K_X^{\otimes m})$. In this aspect, our result can be restated as that the birationality problem for given $X$ and $Y$ is characterised by an explicit countable set of invariants.) 

{\bf Bounded Birationality Problem.} Given $X,Y$ irreducible complex algebraic subvarieties of $\mathbb{P}^n$ (where $n\geq 2$) and a positive integer $d$. Is there a (bi)rational map $F:\mathbb{P}^n\dashrightarrow \mathbb{P}^n$ with $\deg (F)\leq d$ whose restriction to $X$ is a birational map $F|_X:X\dashrightarrow Y$?

Note that any rational map $X\dashrightarrow Y$ is the restriction of a rational map $\mathbb{P}^n\dashrightarrow \mathbb{P}^n$.  Moreover, by results in \cite{mella-polastri, ciliberto-cueto-mella-ranestad-zwiernik}, whenever $X,Y\subset\mathbb{P}^n$ with $n-2\geq \dim (X), \dim (Y)$,  the existence of a birational map $X\dashrightarrow Y$ is equivalent to the existence of a birational map $F:\mathbb{P}^n\dashrightarrow \mathbb{P}^n$ whose restriction to $X$ is a birational map $F|_X:X\dashrightarrow Y$. On the other hand, there are cases of $\dim (X)=\dim (Y)=n-1$ for which there are no such birational map extensions $F$ \cite{mella-polastri2}.  In the above formulation of the Bounded Birationality Problem, we allow $X,Y$ to be hypersurfaces of $\mathbb{P}^n$. 

Define $\mathcal{B}(X,Y,d):=\{F:\mathbb{P}^n\dashrightarrow \mathbb{P}^n:~F$ is {\bf rational}, $\deg (F)\leq d$, $F|_X$ is birational from $X$ onto $Y\}$. Also, define  $\mathcal{B}^+(X,Y,d):=\{F:\mathbb{P}^n\dashrightarrow \mathbb{P}^n:~F$ is {\bf birational}, $\deg (F)\leq d$, $F|_X$ is birational from $X$ onto $Y\}$. The set $\bigcup _{d=1}^{\infty}\mathcal{B}(X,Y,d)$ is exactly those rational selfmaps of $\mathbb{P}^n$ whose restriction to $X$ is a birational map onto $Y$. A similar interpretation can be given for $\bigcup _{d=1}^{\infty}\mathcal{B}^+(X,Y,d)$.   Note that from results in \cite{mella-polastri2} as mentioned above, if $X$ and $Y$ are hypersurfaces in $\mathbb{P}^n$, it may happen that $\bigcup _{d=1}^{\infty}\mathcal{B}(X,Y,d)\not= \emptyset$ while $\bigcup _{d=1}^{\infty}\mathcal{B}^+(X,Y,d)= \emptyset$. On the other hand, by \cite{mella-polastri} as mentioned above, if $\dim (X)=\dim (Y)\leq n-2$, then $\bigcup _{d=1}^{\infty}\mathcal{B}(X,Y,d)\not= \emptyset$ if and only if $\bigcup _{d=1}^{\infty}\mathcal{B}^+(X,Y,d)\not= \emptyset$.

The main result in this section is the following. (An accompanying algorithm will also be given.) 
\begin{theorem}
The Bounded Birationality Problem has a solution if and only if at least one among {\bf explicitly} constructed $C_1$ systems of polynomial equations in $C_2$ variables, each polynomial  of degree bounded from above by a constant $C_3$, has one solution in $\mathbb{C}$. Here $C_1, C_2, C_3$ and the number of polynomials in each system of equations are explicitly bounded in terms of $X, Y$, $\deg (F)$ and $n$.

In other words, there is an explicitly bounded number of variables $\alpha _1,\ldots ,\alpha _N$, an explicitly constructed finite dimensional (generally reducible) variety $W(X,Y,d)$ in variables $\alpha _1,\ldots ,\alpha _N$,  and a surjective map $\kappa  :W(X,Y,d)\rightarrow \mathcal{B}(X,Y,d)$. The variety $W(X,Y,d)$ is non-empty if and only if the set $\mathcal{B}(X,Y,d)$ is non-empty. 

Similarly, there is an explicitly constructed finite dimensional variety $W^+(X,Y,d)$ and a surjective map $\kappa ^+:W^+(X,Y,d)\rightarrow \mathcal{B}^+(X,Y,d)$. 

\label{TheoremMain}\end{theorem}
\begin{remark} From the proof of the theorem, it is easy to see that it is also valid for varieties over an arbitrary algebraically closed field. \end{remark}

We now present a consequence of this result. Denote by $\mathcal{R}(n,d)$ the set $\{F:\mathbb{P}^n\dashrightarrow \mathbb{P}^n:~F$ is {\bf rational}$\}$, and by $\mathcal{R}^+(n,d)$ the set $\{F:\mathbb{P}^n\dashrightarrow \mathbb{P}^n:~F$ is {\bf birational}$\}$. We recall that here we work with Zariski topology, and that a set in a topological space is {\bf locally closed} if it is the intersection of an open and a closed set, and a set is {\bf constructible} if it is  a finite union of locally closed sets.  
 
\begin{corollary}
The sets $\mathcal{R}(n,d)$ and $\mathcal{R}^+(n,d)$ are algebraic varieties.  

The subset $\mathcal{B}(X,Y,d)$ of $\mathcal{R}(n,d)$ is constructible. In particular, $\mathcal{B}(X,Y,d)$ is an algebraic variety, even though it may not be a closed subvariety of $\mathcal{R}(n,d)$.

Similarly, the subset $\mathcal{B}^+(X,Y,d)$ of $\mathcal{R}^+(n,d)$ is constructible. In particular, $\mathcal{B}^+(X,Y,d)$ is an algebraic variety, even though it may not be a closed subvariety of $\mathcal{R}^+(n,d)$. 
\label{CorollaryMain}\end{corollary} 

Before giving the detail of the proofs of the above results, we introduce some useful lemmas. 

The following result, so-called Andreotti-Bezout inequality in the literature (see e.g. \cite{angeniol}), will be used throughout the paper. (A more general result - applied for intersection of general varieties instead of hypersurfaces, and on a general ambient space not necessarily $\mathbb{P}^N$ but with a non-optimal constant $C$ - is also used at several points in the remaining of this paper and can be deduced from Lemma 4.1 in \cite{truong}.)
\begin{lemma}
Let $\mathcal{I}=<f_1,\ldots ,f_k>$ be an ideal in $\mathbb{P}^N$. Assume that the degrees of $f_1,\ldots ,f_k$ are all bounded from above by a given positive integer $d$. Then there is a constant $C>0$, depending only on $N$, so that the following holds. If $V\subset \mathbb{P}^N$ is any irreducible component of the reduced variety defined by $\mathcal{I}$, then $\deg (V)\leq C\deg (f_1)\ldots \deg (f_k)$.   
\label{Lemma0}\end{lemma}
\begin{remark} In fact, in Lemma \ref{Lemma0}, we can choose $C=1$. We formulated the lemma this way in order to comply with the general version of it referred to in the paragraph in front of the statement of the lemma. \end{remark}

Consequently, we obtain effective upper bounds for the degree of the graph of a map of bounded degree and for degrees of linear projections. More precisely, we have:
\begin{lemma}
1) Let $X,Y$ be irreducible subvarieties of $\mathbb{P}^n$. Assume that a rational map $f:X\dashrightarrow Y$ is given of degree $\leq d$, that is it is the restriction of a rational self-map of $\mathbb{P}^n$ of degree $\leq d$. Let $\Gamma _f\subset \mathbb{P}^n\times \mathbb{P}^n$ be the graph of $f$. Then the degree of $\Gamma _f$ is effectively bounded in terms of $d$.

2) Let $\pi :\mathbb{C}^n\rightarrow \mathbb{C}^{n-1}$ be the natural projection. Let $Z\subset \mathbb{C}^n$ be an irreducible variety. Then there is a constant $C>0$, independent of $Z$, so that $\deg (\pi (Z))\leq C\deg (Z)$. 
\label{LemmaDegreeGraph}\end{lemma}
\begin{proof}
1) We can proceed as follows. Let $x_0,x_1,\ldots, x_n$ be the homogeneous coordinates for the copy of $\mathbb{P}^n$ containing $X$, and $y_0,y_1,\ldots ,y_n$ be the homogeneous coordinates for the copy of $\mathbb{P}^n$ containing $Y$. Let $F=[F_0:\ldots :F_n]$ be a rational map from $\mathbb{P}^n$ to itself whose degree is $\leq d$ and whose restriction to $X$ is $f$. Then $F_i$'s are homogeneous polynomials of degree $\leq d$ in $x_0,\ldots ,x_n$. Hence, the graph $\Gamma _F$ is a component where the intersection of the $n$ hypersurfaces $y_iF_j-y_jF_i=0$ is {\bf proper} - that is of correct dimension - where the indices run on all $i,j$ for which both $F_i$ and $F_j$ are non-zero. Since the degrees of these hypersurfaces are bounded explicitly in terms of $d$, and the number of these hypersurfaces is also explicitly bounded, we have that the degree of $\Gamma _F$ is explicitly bounded, by applying Lemma \ref{Lemma0}. Next, since $\Gamma _f$ is one component where $\Gamma _F$ and $X\times Y$ intersect properly, it follows from Lemma 4.1 in \cite{truong} that the degree of $\Gamma _f$ is also explicitly bounded.

2) Let $\overline{Z}\subset \mathbb{P}^n$ be the closure of $\overline{Z}$. Then $\overline{Z}$ has the same degree as that of $Z$.  Similarly, if $\overline{\pi (Z)}\subset \mathbb{P}^{n-1}$, then $\overline{\pi (Z)}$ has the same degree as that of $\pi (Z)$. Let $H'\subset \mathbb{P}^n$ and $H\subset \mathbb{P}^{n-1}$ be generic hyperplanes. Let $k=\dim (Z)$ and $l=\dim (\pi (Z))$. Then $(H')^{k-l}$ and $\overline{Z}$ intersect properly, and $\pi ((H')^{k-l}.\overline{Z})$ contains $\overline{\pi (Z)}$. We have 
\begin{eqnarray*}
\deg (\overline{\pi (Z)})=H^{l}.\overline{\pi (Z)}\leq H^l. \pi _*((H')^{k-l}.\overline{Z})=\pi _*((H')^{k-l}.\pi ^*(H).\overline{Z}). 
\end{eqnarray*}
From this, we obtain the desired conclusion $\deg (\pi (Z))\leq C\deg (Z)$. 
\end{proof}

For the next lemma, we first introduce the notation and some basic properties of monoids, which are important in the proof. For more detail, see for example Sections 1.2 and 1.3 in \cite{ciliberto-cueto-mella-ranestad-zwiernik}.  We use the notation $(0)_p=(0,\ldots ,0)$ ($p$ times). We also use combinations of these, such as $((0)_p,1,(0)_q)$ and $((0)_p,(0)_q,1)=((0)_{p+q},1)$. 

A hypersurface $M\subset \mathbb{P}^r$ (where $r\geq 2$) of degree $d$ is a monoid with vertex $p$ if it is {\bf irreducible} and $p$ is a point in $M$ of multiplicity exactly $d-1$. If we choose the coordinates for $\mathbb{P}^r$ so that $p=[0:\ldots :0:1]$, then the defining equation for $M$ is $f_{d-1}(x_0,\ldots ,x_{r-1})x_r+f_d(x_0,\ldots ,x_{r-1})$. Here $f_{d-1}$ is a homogeneous polynomial of degree $d-1$ and $f_{d}$ is a homogeneous polynomial of degree $d$. The assumption that $M$ is irreducible is equivalent to that $GCD(f_{d-1},f_d)=1$.      

The projection from $p$ gives rise to a birational map between a monoid $M\subset \mathbb{P}^r$ and the projective space $\mathbb{P}^{r-1}$. More precisely, the projection map is $\pi :M\dashrightarrow \mathbb{P}^{r-1}$ given by $[x_0:\ldots :x_{r-1}:x_r]\mapsto [x_0:\ldots :x_{r-1}]$, and the inverse $\pi ^{-1}:\mathbb{P}^{r-1}\dashrightarrow M$ is given by $$[x_0:\ldots :x_{r-1}]\mapsto [f_{d-1}x_0:\ldots :f_{d-1}x_{r-1}:-f_d].$$
From these formulas, it is obvious that $\pi$ maps $M\backslash \{f_{d-1}=f_d=0\}=M\backslash \{f_{d-1}=0\}$ isomorphically to its image $\mathbb{P}^{r-1}\backslash \{f_{d-1}=0\}$. We note also that the indeterminacy set of $\pi $ is $p$, and the indeterminacy set of $\pi ^{-1}$ is $\{f_{d-1}=f_d=0\}$. Since $M$ is the cone over $\{f_{d-1}=f_d=0\}$ with vertex $p$, it follows that for a subvariety $p\not= A\subset M$, the cone $C_p(A)$ over $A$ with vertex $p$ is contained in $M$ iff the image $\pi (A\backslash\{p\} )\subset \mathbb{P}^{r-1}$ is in the indeterminacy set of $\pi ^{-1}$. 

Note that a monoid $M$ can have two or more vertices (for example, the monoid $M_{0,0}$ in Step 3 of the proof below). If a monoid $M\subset \mathbb{P}^r$ has two vertices $[0:\ldots :0:1]$ and $[0:\ldots :0:1:0]$, then it has a defining equation of the form: $$f_{d}(x_0,\ldots ,x_{r-2})+x_{r-1}g(x_0,\ldots ,x_{r-2})+x_rh_{d-1}(x_0,\ldots ,x_{r-2})+x_rx_{r-1}f_{d-2}(x_0,\ldots ,x_{r-2})=0.$$  
In this case, if we compose the birational map $\mathbb{P}^{r-1}\dashrightarrow M$ of the projection from one vertex and the projection $M\dashrightarrow \mathbb{P}^{r-1}$ from the other vertex of $M$, then we obtain a biration map $\mathbb{P}^{r-1}\dashrightarrow \mathbb{P}^{r-1}$.  Moreover, we note that in this case the monoid is irreducible, as an algebraic variety, iff $GCD(f_{d}(x_0,\ldots ,x_{r-2})+x_{r-1}g(x_0,\ldots ,x_{r-2}), h_{d-1}(x_0,\ldots ,x_{r-2})+x_{r-1}f_{d-2}(x_0,\ldots ,x_{r-2}))=1$ (apply the case of monoids with only one vertex above). 

The following effective versions of Lemma 2.1 and 2.2 in \cite{ciliberto-cueto-mella-ranestad-zwiernik} are also important in the proof of Theorem  \ref{TheoremMain}. The fact that the bound is independent of the vertex $p$ is crucial.
\begin{lemma}
1) Let $Z\subset \mathbb{P}^r$, with $r\geq 3$, be an irreducible variety of dimension $1\leq s\leq r-2$ and let $p\in \mathbb{P}^r$ be such that the projection of $Z$ from $p$ is birational to its image. Then for $d$ large enough, depending only on {\bf the degree} of $Z$ and $r$, there is a monoid in $\mathbb{P}^r$ of degree $d$ with vertex $p$, containing $Z$ but not containing the cone $C_p(Z)$ over $Z$ with vertex $p$. 

2) Let  $Z\subset \mathbb{P}^r$, with $r\geq 3$, be an irreducible variety of dimension $1\leq s \leq r-3$, and let $p_1,p_2\in \mathbb{P}^r$ be distinct points such that the projection of $Z$ from the line going through $p_1$ and $p_2$ is birational to its image. Then, with the same number $d$ as in part 1, there is a monoid in $\mathbb{P}^r$ of degree $d$ with vertices $p_1$ and $p_2$, containing $Z$ but not containing the cones $C_{p_i}(Z)$ over $Z$ with vertices $p_i$ (for $i=1,2$).    
\label{LemmaMonoidals}\end{lemma}
 \begin{proof}
 1) Let $\pi :V\rightarrow \mathbb{P}^r$ be the blowup of $\mathbb{P}^r$ at $p$. Let $Z'\subset V$ be the strict  transform of $Z$. From the proof of Lemma 2.1 in \cite{ciliberto-cueto-mella-ranestad-zwiernik}, it suffices to show the following:  $h^0(Z',O_{Z'}(d(H-E)))$ is bounded explicitly in terms of $d$ and $(H-E)^{s}\cdot Z'$. (The latter is smaller than or equal to the degree of $Z$. To see this claim about the bound for $(H-E)^{s}\cdot Z'$, we expand (here $C(s,j)$ are binomial numbers)
 \begin{eqnarray*}
 (H-E)^{s}\cdot Z'&=&H^{s}\cdot Z'+\sum _{j=1}^{s}(-1)^jC(s,j)H^{s-j}\cdot E^{j}\cdot Z'\\
 &=&\deg (Z)+\sum _{j=1}^{s}(-1)^jC(s,j)H^{s-j}\cdot E^{j}\cdot Z',
 \end{eqnarray*}
 and note that for all $j\geq 1$ the class of $E^j$ is $(-1)^{j-1}$ of a linear subspace in $E$, and the intersections between linear subspaces of $E$ and $Z'$ are psef. Therefore, $\sum _{j=1}^{s}(-1)^jH^{s-j}\cdot E^{j}\cdot Z'\leq 0$ and $ (H-E)^{s}\cdot Z'\leq \deg (Z)$.)  
 
 More precisely, putting $\delta =(H-E)^{s}.Z'\geq 1$, we will show the following: There is a polynomial $p_{s,\delta }(d)$ in the variable $d$, of degree $s$, whose all coefficients are explicitly bounded in terms of $s$ and $\delta$, whose coefficient of $d^{s}$ is $\delta /s!$, and so that $h^0(Z',O_{Z'}d(H-E))\leq p_{s,\delta }(d)$ for all non-negative integers $d\geq 0$.  
 
 Note that since the linear system  $H-E$ is nef and big, and is also movable (having no base locus, this can be seen by observing that the strict transform in $V$ of a hyperplane in $\mathbb{P}^r$ containing $p$ is an element of the linear system $|H-E|$), we have a SES
 \begin{eqnarray*}
 0\rightarrow O_{Z'}(d(H-E))\rightarrow O_{Z'}((d+1)(H-E))\rightarrow O_{Z_1}((d+1)(H-E))\rightarrow 0,
 \end{eqnarray*}
 where $Z_1=Z'\cap S$ ($S$ is a generic element of $|H-E|$) is a variety of dimension $=$ $\dim (Z')-1$ $=s-1$, and $(H-E)^{s-1}.Z_1=(H-E)^{s}.Z'$. Therefore, we have a LES
 \begin{eqnarray*}
 0&\rightarrow& H^0(O_{Z'}(d(H-E)))\stackrel{i_1}{\rightarrow} H^0(O_{Z'}((d+1)(H-E)))\stackrel{i_2}{\rightarrow} H^0(O_{Z_1}((d+1)(H-E)))\\
 &\rightarrow& H^1(O_{Z'}(d(H-E)))\rightarrow \ldots 
 \end{eqnarray*}
 From this we obtain 
 \begin{eqnarray*}
 h^0(O_{Z'}((d+1)(H-E)))&=&\dim (ker (i_2))+\dim (im (i_2))\\
 &=&\dim (im (i_1))+\dim (im (i_2))\\
   &\leq& h^0(O_{Z'}{d(H-E)})+h^0(O_{Z_1}(d+1)(H-E)).
 \end{eqnarray*}
 
 Summing this over $d$, we obtain 
 \begin{eqnarray*}
 h^0(O_{Z'}d(H-E))\leq \sum _{j=0}^{d}h^0(O_{Z_1}j(H-E)).
 \end{eqnarray*}
 
 Then we have the conclusion by induction on the dimension of $Z_1$. In fact, the base case of dimension $0$ is obvious, in which case we obtain that $Z_1$ is a union of $\delta $ points, and hence $h^0(O_{Z_1}d(H-E))=\delta $ for all non-negative integers $d$. If we define by induction the following sequence of polynomials: $q_{0}(d)=1$ for all $d$, and for $m\geq 1$ 
 \begin{eqnarray*}
 q_m(d)=\sum _{j=0}^dq_{m-1}(j),
 \end{eqnarray*}
 then we have that $q_m(d)$ is a polynomial of degree $m$ in $d$ whose coefficient of $d^m$ is $1/m!$. (The first several elements in the sequence are listed below: 
 \begin{eqnarray*}
 q_0(d)&=&1,\\
 q_1(d)&=&d+1,\\
 q_2(d)&=&\frac{d(d+1)}{2}+d+1.)
 \end{eqnarray*}
Moreover, we have from the above arguments that 
\begin{eqnarray*}
h^0(Z',O_{Z'}d(H-E))\leq \delta q_{s}(d)
\end{eqnarray*}
for all non-negative integers $d$. Therefore, the choice of $p_{s,\delta }(d)=\delta q_{s}(d)$ satisfies the claim. 

2) This follows from 1) as in the proof of Lemma 2.2 in \cite{ciliberto-cueto-mella-ranestad-zwiernik}. In fact, let $H_1$ and $H_2$ be the hyperplanes in $\mathbb{P}^r$ so that the projection from $p_1$ maps onto $H_1$ and the projection from $p_2$ maps onto $H_2$. Let $Z_1\subset H_1$ and $Z_2\subset H_2$ be images of $Z$ under the mentioned projections. Then $\deg (Z_1), \deg (Z_2)\leq \deg (Z)$, being (a component of) the (proper) intersection between a cone over $Z$ and $H_1$ (respectively $H_2$). Since $\dim (Z)\leq r-3$, we have that $\dim (Z_1)\leq \dim (H_1)-2$ and $\dim (Z_2)\leq \dim (H_2)-2$. Moreover, if $p_1'$ is the image of $p_1$ under the projection from $p_2$, then the projection under $p_1'$ of $Z_2$ is birational to its image. Similarly for the image $p_2'$ of $p_2$ under the projection from $p_1$.  We can then apply part 1 to obtain a monoid $S_1'$ in $H_1$ of degree $d$ with vertex $p_2'$,  containing $Z_1$ but not the cone $C_{p_2'}(Z_1)$. Let $S_1$ be the cone over $S_1'$ with vertex $p_1$. Then $S_1$ is a monoid of the same degree $d$ whose vertices include all points on the line contacting $p_1$ and $p_2'$, and hence in particular include $p_1$ and $p_2$. Moreover, $S_1$ contains $Z$ but not the cone over $Z$ with vertex $p_2$. We construct a similar monoid $S_2$. Then a generic linear combination between $S_1$ and $S_2$ gives the desired answer.        
\end{proof}

Here is the proof of Theorem \ref{TheoremMain}. It gives an algorithm to construct the systems of polynomial equations. For the convenience of the readers, we will present the algorithm explicitly afterwards.  

\begin{proof}[Proof of Theorem \ref{TheoremMain}] We first give the conclusion for the set $\mathcal{B}(X,Y,d)$. By a linear change of coordinates, we may assume that none of $X,Y$ belongs to the hyperplane at infinity of $\mathbb{P}^n$.  
For simplicity, from now on (except when otherwise indicated) we will work with the Zariski open dense sets $\mathbb{C}^k$ of $\mathbb{P}^k$ only. 

{\bf Step 1} (Forward argument):  If $X=Y=\mathbb{P}^n$ then there is nothing to do. Hence we can assume that $\dim (X)=\dim (Y)\leq n-1$. Since $n\geq 2$ and $\dim (\Gamma _f)=\dim (X)$, it follows that $\Gamma _f\subset \mathbb{C}^n\times \mathbb{C}^n$ is of codimension at least $3$. Hence, by Lemma \ref{LemmaMonoidals}, there is a monoid $M_{0,0}\subset \mathbb{C}^n\times \mathbb{C}^n$  of degree effectively bounded by the degree of $\Gamma _f$, containing $\Gamma _f$ and has two vertices $((0)_{n-1},1,(0)_n)$ and $((0)_n, (0)_{n-1},1)$. Note that the projections from  $((0)_n, (0)_{n-1},1)$ and $(1,(0)_{n-1},(0)_n)$ are birational from $\Gamma _f$ onto its image. Let $\Gamma _{0,1}\subset \mathbb{C}^n\times \mathbb{C}^{n-1}$ and $\Gamma _{1,0}\subset \mathbb{C}^{n-1}\times \mathbb{C}^n$ be the corresponding images of $\Gamma _f$. Their degrees are effectively bounded in terms of the degree of $\Gamma _f$ (see part 2 of Lemma \ref{LemmaDegreeGraph}) and hence, by part 1 of Lemma \ref{LemmaDegreeGraph}, in terms of $d$. 

 Similarly, we can find a monoid $M_{0,1}\subset \mathbb{C}^n\times \mathbb{C}^{n-1}$ of degree effectively bounded by the degree of $\Gamma _{0,1}$ and has the vertex $((0)_n, (0)_{n-2},1)$. Again, note that the projection from the point $((0)_n, (0)_{n-2},1)$ is birational from $\Gamma _{0,1}$ to its image $\Gamma _{0,2}\subset \mathbb{C}^n\times \mathbb{C}^{n-2}$. Similarly, we construct $M_{1,0}\subset \mathbb{C}^{n-1}\times \mathbb{C}^n$ and $\Gamma _{2,0}\subset \mathbb{C}^{n-2}\times \mathbb{C}^n$. 

Repeating the above process, we obtain monoids $M_{0,j}\subset \mathbb{C}^n\times \mathbb{C}^{n-j}$, $M_{j,0}\subset \mathbb{C}^{n-j}\times \mathbb{C}^n$ and the birational images of $\Gamma _f=\Gamma _{0,0}$: they are $\Gamma _{0,j}\subset \mathbb{C}^n\times \mathbb{C}^{n-j}$ and $\Gamma _{j,0}\subset \mathbb{C}^{n-j}\times \mathbb{C}^{n}$. Note that $\Gamma _{0,n}=X$ and $\Gamma _{n,0}=Y$. 

{\bf Step 2} (Backward argument - The main part of the proof) Now to prove Theorem \ref{TheoremMain}, we {\bf reverse} the above argument, and proceed as follows. 

We do not know the graph $\Gamma _f$ and consequently, the monoids $M_{0,j}$ and $M_{j,0}$ above. Also we do not know the intermediate $\Gamma _{j,0}$ and $\Gamma _{0,j}$, except that $\Gamma _{0,n}=X$ and $\Gamma _{n,0}=Y$. 

We do, however, {\bf know} that the degrees of the monoids $M_{0,j}$ and $M_{j,0}$ are effectively bounded. We also know that these monoids are in fixed variables $x_1,\ldots ,x_n, y_1,\ldots ,y_{n-j}$ and $x_1,\ldots ,x_{n-j},y_1,\ldots ,y_n$. Therefore, we can parametrise them in terms of their coefficients. The number of these coefficients is effectively bounded. 

Now we start the process of writing out the equations promised in the theorem. We will start from X, and then go step by step through all $\Gamma _{0,j}$ and $\Gamma _{j,0}$, at each step add some polynomial equations, and end up at $Y$. At that stage we have the needed polynomial systems.

 We start from $X=\Gamma _{0,n}\subset \mathbb{C}^n$ and want to go up to $\Gamma _{0,n-1}\subset \mathbb{C}^n\times \mathbb{C}$. We do not know about $\Gamma _{0,n-1}$, but we know that $\Gamma _{0,n-1}\subset M_{0,n-1}$ and $X$ is the birational image of $\Gamma _{0,n-1}$ via the projection from the point $((0)_n,1)$. We next show that this can be described in terms of some polynomial equations.

Using the homogeneous coordinates, the monoid $M_{0,n-1}$ is given by an equation $$f_{0,n-1}(x_0,\ldots ,x_n,y_1)=f_{0,n-1,1}(x_0,\ldots ,x_{n})y_1+f_{0,n-1,2}(x_0,\ldots ,x_n)=0.$$ We can dehomogenise to get an equation in affine coordinates $x_1,\ldots ,x_n,y_1$. Knowing the bound on the degree of $f_{0,n-1}$, we know how many parameters (which are the coefficients of the monomials in the equation for the given monoid) we will need. 

 From Lemma \ref{LemmaMonoidals}, we see that $X$ is the birational image of some $\Gamma _{0,n-1}\subset M_{0,n-1}$ under the projection from $((0)_n,1)$ exactly when $X\not\subset \{f_{0,n-1,1}=f_{0,n-1,2}=0\}$. We note that in $M_{0,n-1}$, {\bf set theoretically} the locus $\{f_{0,n-1,1}=f_{0,n-1,2}=0\}$ is the same as the hypersurface $\{f_{0,n-1,1}=0\}$. Provided this condition is satisfied, then $\Gamma _{0,n-1}$ will be the {\bf strict transform} in $M_{0,n-1}$ of $X$, under the inverse of the projection map $\pi$. Hence $\Gamma _{0,n-1}$ is contained in the total inverse image $\pi ^{-1}(X)\cap M_{0,n-1}$. If $X$ is defined by an ideal $I(X)=\{g_1,\ldots ,g_m\}$, then the total inverse image of $X$ is given by the ideal $\{g_1,\ldots ,g_m, f_{0,n-1,1}y_1+f_{0,n-1,2}\}$. However, this set usually is bigger than what we want (the variety $\Gamma _{0,n-1}$), and it will make later computations and arguments harder. So, one {\bf key idea} here is to consider only the preimage $H_{0,n-1}$ of the Zariski open set $X\backslash \{f_{0,n-1,1}=0\}$ of $X$, which is given by the ideal $\{g_1,\ldots ,g_m,f_{0,n-1,1}y_1+f_{0,n-1,2},1-t_{0,n-1}f_{0,n-1,1}\}$. (Using the common trick, we added a variable $t_{0,n-1}$.)  This latter set $H_{0,n-1}$ is a Zariski open dense set of the hypothetical $\Gamma _{0,n-1}$, and it is isomorphic (to see this, note that the projection map from $((0)_{n},1)$ is an isomorphism between $M_{0,n-1}\backslash \{f_{0,n-1}=0\}$ onto its image)  to the Zariski dense open set $X\backslash \{f_{0,n-1,1}=0\}$ of $X$, provided that $X\backslash \{f_{0,n-1,1}=0\}$ is non-empty and the following conditions on the monoid are satisfied. 

\begin{itemize}
\item The first condition is that both $f_{0,n-1,1}$ and $f_{0,n-1,2}$ are non-zero polynomials. The condition that $f_{0,n-1,1}$ is a non-zero polynomial is already taken care (provided $H_{0,n-1}$ is non-empty) by the equation $1-t_{0,1}f_{0,n-1,1}=0$ in the defining equations for $H_{0,n-1}$. For the condition that $f_{0,n-1,2}$ is non-zero, we need only that at least one of the coefficients $a_I$ of some monomial $x^I$ is non-zero, and this condition can be again described using the trick of adding one new variable $s_I$ so that $1-s_Ia_I=0$. So an explicitly bounded number of such equations will cover our case. 

\item The second condition is that  the monoid $M_{0,n-1}$ should be {\bf irreducible}, which is the same as that $GCD(f_{0,n-1,1},f_{0,n-1,2} )=1$. We will show that this condition is also taken care by the equation $1-t_{0,n-1}f_{0,n-1,1}=0$ already given above. In fact, assume that $GCD(f_{0,n-1,1},f_{0,n-1,2})=h_{0,n-1}$. Then, we can write $f_{0,n-1,1}=h_{0,n-1}w_{0,n-1,1}$ and $f_{0,n-1,2}=h_{0,n-1}w_{0,n-1,2}$ where $GCD(w_{0,n-1,1},w_{0,n-1,2})=1$. Then, the equation $$0=1-t_{0,n-1}f_{0,n-1,1}=1-t_{0,n-1,1}h_{0,n-1,1}w_{0,n-1,1}$$ of $H_{0,n-1}$ implies that in fact $H_{0,n-1}$ is contained in the set $\{h_{0,n-1,1}\not= 0\}$, and hence is contained in the irreducible monoid $w_{0,n-1,1}y_{1}+w_{0,n-1,2}=0$.    
\end{itemize}

Iterating this argument, we can go back further and introduce some explicitly bounded number of systems of equations (each time at one monoid), to go back from $X\subset \mathbb{C}^n$ to some $H_{0,1}\subset \mathbb{C}^n\times \mathbb{C}^{n-1}$. By a similar argument (but a bit more complicated, see below), we go back to some $H_{0,0}\subset \mathbb{C}^n\times \mathbb{C}^n$. $H_{0,0}$ is then a Zariski dense open set of the hypothetical graph $\Gamma _f$. The difference between $H_{0,0}$ and the other $H_{0,j}$ (where $1\leq j\leq n-1$) is that while the other $H_{0,j}$'s belong to monoids with only one vertex, $H_{0,0}$ belongs to a monoid with 2 vertices. The equation for such a monoid is 
\begin{eqnarray*}
&&f_{0,0,1}(x_1,\ldots ,x_{n-1},y_1,\ldots ,y_{n-1})+x_n g_{0,0}(x_1,\ldots ,x_{n-1},y_1,\ldots ,y_{n-1})\\
&&+y_nh_{0,0}(x_1,\ldots ,x_{n-1},y_1,\ldots ,y_{n-1})+x_ny_nf_{0,0,2}(x_1,\ldots ,x_{n-1},y_1,\ldots, y_{n-1})\\
&=&f_{0,0,1}+x_ng_{0,0}+y_n(h_{0,0}+x_nf_{0,0,2})\\
&=&f_{0,0,1}+y_nh_{0,0}+x_n(g_{0,0}+y_nf_{0,0,2}).
\end{eqnarray*}
Here the last two equalities express the same monoid regarded as a monoid of either one of the two vertices in concern. As before, $H_{0,0}$ is not contained in either of the bad sets of the projections from these two vertices, that is $H_{0,0}\not\subset \{f_{0,0,1}+x_ng_{0,0}=0\}\cup \{f_{0,0,1}+y_nh_{0,0}+x_n(g_{0,0}=0\}$. As before, we can consider only the complement in of the latter set, which can be described in terms of polynomial equations: $1-t_{0,0}(f_{0,0,1}+x_ng_{0,0})(f_{0,0,1}+y_nh_{0,0}+x_n(g_{0,0})=0$, where $t_{0,0}$ is a new variable.

Now, we see that the graph $\Gamma _f$, and hence the birational map $f$, exists iff the following two conditions are satisfied: 

\begin{itemize}
\item Condition 1: At least one among the many systems of equations which we produced for $H_{0,0}$ in the above has a non-empty solution set. (Each such solution corresponds to one $H_{0,0}$, but the correspondence may not be 1-to-1.) If this is the case, then from the construction it is obvious that $H_{0,0}$ is birational to $X$. The latter means  exactly that $H_{0,0}$ is the graph of a rational map from $X$ into $\mathbb{P}^n$. 

\item Condition 2: For at least one system of equations in Condition 1 which has a non-empty solution set, the corresponding $H_{0,0}$ is birational to $Y$ via the second projection $\mathbb{C}^n\times \mathbb{C}^n\rightarrow \mathbb{C}^n$. This means exactly that $H_{0,0}$ is the graph of a birational map from $X$, and the image of that map is $Y$. 
\end{itemize}

Towards Condition 2, we go step by step, using monoids on the $Y$-side: $M_{0,0}$, $M_{1,0}$ and so on. 

We can apply the same argument employed when we went up from $X=\Gamma _{0,n}$ to $\Gamma _{0,n-1}$. Here we describe how to go from $H_{0,0}\subset \mathbb{C}^n\times \mathbb{C}^n$ to $H_{1,0}\subset \mathbb{C}^n\times \mathbb{C}^n$ (important note: $H_{1,0}$ is still in $\mathbb{C}^n\times \mathbb{C}^n$), the latter being isomorphic to a dense set of the image of $H_{0,0}$ in $\mathbb{C}^{n-1}\times \mathbb{C}^n$ under the projection from $(1,(0)_{n-1},(0)_n)$. (Note {\bf the  difference} here is that we do not need to compute the image of $H_{0,0}$ under the concerned linear projection, which is quite complicated - for example not Zariski closed - in particular when we have parameters in the defining equations.) This $H_{1,0}$ is isomorphic to its image via the projection from the point $(1,(0)_{n-1},(0)_n)$, and its image is contained in $M_{1,0}$. We will, as when going up from $X$ to $\Gamma _{0,n-1}$, consider only a Zariski open set of $H_{0,0}$ lying inside the monoid $M_{0,0}$ where the projection is an isomorphism to its image. We also need to check that the image is contained in the monoid $M_{1,0}$, and this can be done by checking that for all polynomials $h$ in a defining ideal of $M_{1,0}$, the hypersurface $h=0$ (now considered as a hypersurface in $\mathbb{C}^n\times \mathbb{C}^n$ in stead of $\mathbb{C}^{n-1}\times \mathbb{C}^n$) contains the variety $H_{0,0}$. (In effect, we are checking that the preimage under the projection from $(1,(0)_{n-1},(0)_n)$ of $M_{1,0}$ contains as a set the variety $H_{0,0}$). We will show that this can be done by adding some polynomials, as below in the check of whether the image of $H_{n,0}$ belongs to $Y$.    

At the end of this induction process, we get some $H_{n,0}\subset \mathbb{C}^n\times \mathbb{C}^n$ which is birational to $H_{0,0}$ and hence to $X$. We need only to check that the image of $H_{n,0}$ is  a subset of $Y$ (since then, because $H_{n,0}$ and $Y$ have the same dimension, and $H_{n,0}$ is birational to its image, we get that $H_{n,0}$ is birational equivalent to $Y$, and are done). 

Now, we proceed to show that whether the image of $H_{0,0}$ under the projection to the second factor $\mathbb{C}^n\times \mathbb{C}^n\rightarrow \mathbb{C}^n$ belongs to $Y$ is determined by some polynomial equations. This is the same  as requiring the preimage of $Y$ under this projection contains $H_{n,0}$. This reduces to the following question: Let $H_{n,0}$ be defined by an ideal $\{h_1,\ldots ,h_k\}$ in variables $x_1,\ldots ,x_n,y_1,\ldots ,y_n, s_I$  (depending on parameters as well). Let $h$ be a polynomial (which we can think as an element in the ideal of $Y$). We need  to show that whether the set $\{h_1=\ldots =h_k=0\}$ is contained in $h=0$ is described by some polynomial equations. 

To this end, we use the following common method. Working in a ring $\mathbb{C}[w_1,\ldots ,w_N]$. The set $\{h_1=\ldots =h_k=0\}$ is contained in $h=0$ iff in $\mathbb{C}[w_1,\ldots ,w_N,a]$, where $a$ is a new variable, the function $1$ is in the ideal $<h_1,\ldots ,h_k,1-ah>$. For the readers' convenience, we recall here the classical argument. By Hilbert's Nullstellensatz, $\{h_1=\ldots =h_k=0\}$ is contained in $\{h=0\}$ iff there is some positive integer $j$ so that in $\mathbb{C}[w_1,\ldots ,w_N]$ the polynomial $h^j$ is in the ideal $<h_1,\ldots ,h_k>$. Then in the polynomial ring $\mathbb{C}[w_1,\ldots ,w_N,a]$
\begin{eqnarray*}
1=(1-a^jh^j)+a^jf^j=(1-ah)(1+ah+a^2h^2+\ldots +a^{j-1}h^{j-1})+a^jh^j
\end{eqnarray*}
is in the ideal $<h_1,\ldots ,h_k,1-ah>$.) By effective Hilbert's Nullstellensatz \cite{hermann, brownawell, kollar}, there will be polynomials $\tau _1,\ldots ,\tau _k, \tau \in \mathbb{C}[w_1,\ldots ,w_N,a]$ with explicitly bounded degrees so that 
\begin{eqnarray*}  
1\equiv \tau _1h_1+\ldots +\tau _kh_k+\tau (1-ah)
\end{eqnarray*}
in the ring $\mathbb{C}[w_1,\ldots ,w_N,a]$. Note that $\tau _1,\ldots ,\tau _k$ will have coefficients not yet determined, and we determine them by balancing the coefficients of the polynomials on the two sides of the above identity. We then get a system of polynomial equations in the coefficients of the concerned polynomials, which we need to have at least a solution. 

Combining the above steps, we have several systems of polynomials in the parameter spaces (the coefficients of the monoids $M_{i,j}$ and the coefficients of the polynomials $\tau _i$), at least one among them has a solution if we have a birational map of degree $\leq d$ from $X$ to $Y$. Conversely, any such a solution will provide us with a birational map from $X$ to $Y$ (even though may be of degrees much bigger than $d$). 

The only thing needs to check is whether the above constructed birational maps (corresponding to these solutions) are of degree $\leq d$. That is, it remains for us to check whether at least one of these birational maps from $X\dashrightarrow Y$ comes from a global rational map $F:\mathbb{P}^n\dashrightarrow \mathbb{P}^n$ of degree  $\leq d$ actually. We show that this can be detected by some polynomial equations and proceed as follows. That a birational map $X\dashrightarrow Y$ which we constructed above comes from a rational map $F:\mathbb{P}^n\dashrightarrow \mathbb{P}^n$ is the same as that the variety $H_{0,0}\subset \mathbb{C}^n\times \mathbb{C}^n$ constructed above is contained in the graph $\Gamma _F$. Writing in the affine set $\mathbb{C}^n\times \mathbb{C}^n$ (recalling our convenience from the beginning of this proof), the ideal of $\Gamma _F$ is $<y_1h(x_1,\ldots ,x_n)-g_1(x_1,\ldots ,x_n),\ldots ,y_nh(x_1,\ldots ,x_n)-g_n(x_1,\ldots ,x_n)>$, where $g_1,\ldots ,g_n$ and $h$ are polynomials in $\mathbb{C}[x_1,\ldots ,x_n]$ of degree $\leq d$. Again, these polynomials are not yet determined, but we can parametrise them in terms of their coefficients, and the number of these coefficients are efficiently bounded by the degree $d$. We can check whether $\Gamma _F$ contains $H_{0,0}$ by using the procedure above in checking whether the image of $X$ is contained in $Y$. This creates new polynomial equations. We also need to check that $X$ is not contained in the indeterminacy locus of $F$, which is the same as $H_{0,0}\backslash \{h(x_1,\ldots ,x_n)=0\}$ is non-empty. Adding a new variable $t$, the latter is the same as the following conclusion: the system of polynomial equations $\{H_{0,0},1-th(x_1,\ldots ,x_n)=0\}$ has at least one solution. 

We add all of these new equations into the equations we already constructed above. Therefore, by taking the zero set of these systems of polynomial equations, we have a finite dimensional variety $W(X,Y,d)$ (depending on the variables $x_1,\ldots ,x_n,y_1,\ldots ,y_n$ as well as coefficients of the involved monoids and some other parameters, as stipulated above) together with the surjective map $\kappa$ as in the theorem. A priori, two distinct points in $W(X,Y,d)$ may give rise to the same birational map $X\dashrightarrow Y$. We can also use Gr\"obner bases \cite{gunther, buchberger} to eliminate the variables $x_1,\ldots ,x_n,y_1,\ldots ,y_n$ from $W(X,Y,d)$ to obtain systems of equations in the parameters only, and hence obtain a (generally reducible) variety $\widetilde{W}(X,Y,d)$ of finite dimension in parameters. It is clear that the parameters which give us birational maps between $X$ and $Y$ is a dense subset $\widetilde{W_0}(X,Y,d)$ of the variety $\widetilde{W}(X,Y,d)$.  

To finish off, we give the proof for the conclusion concerning the set $\mathcal{B}^+(X,Y,d)$. Our starting point is the equations in the previous paragraph. We only need to add the equations which the coefficients of a rational map $F:\mathbb{P}^n\dashrightarrow \mathbb{P}^n$ in the above must satisfy in order for $F$ to be a birational map. This can be done similarly to the above proof, with the following modifications. Here we construct a monoid $M'_{0,0}$ for  $\Gamma _F$ with {\bf only one vertex} $(1,(0)_{n-1},(0)_n)$. This is because we only need to check that the projection to the second factor $\mathbb{C}^n\times \mathbb{C}^n$ will give us a birational map, and hence we do not need to consider the vertex $((0)_n,(0)_{n-1},1)$ as before. Since $\Gamma _F$ is of codimension $n\geq 2$ in $\mathbb{C}^n\times \mathbb{C}^n$, we can apply Lemma \ref{LemmaMonoidals}. We need to add the equations which check that the monoid $M'_{0,0}$ contains $\Gamma _F$. Then we can keep follow the argument we gave for $H_{0,0}$, until we reach the isomorphic image $H'_{n-1,0}$ in $\mathbb{C}\times \mathbb{C}^n$ of a Zariski open set of $\Gamma _F$. At this step, we cannot apply Lemma \ref{LemmaMonoidals}, since the dimension of $\Gamma _F$ is $n$ and the dimension of $\mathbb{C}\times \mathbb{C}^n$ is $n+1$. To be able to finish this last step, we simply add one more dimension so that the assumptions in Lemma \ref{LemmaMonoidals} are satisfied, concerning $\mathbb{C}\times \mathbb{C}^n$ as a hyperplane $w=0$ in a bigger space $\mathbb{C}\times \mathbb{C}^n\times \mathbb{C}_w$. Then we can construct a monoid in $\mathbb{C}\times \mathbb{C}^n\times \mathbb{C}_w$ for $H'_{n-1,0}$, and then the proof is finished by checking that $\mathbb{C}^n=\mathbb{C}^n\times \{w=0\}\subset \mathbb{C}^n\times \mathbb{C}_w$ contains the image of $H'_{n-1,0}$ under the projection from the point $(1,(0)_n,w=0)$.  

{\bf Addendum.}  Note that by a linear change of coordinates, we can assume that $X$ and $Y$ are not contained in any coordinate hyperplane. If in Step 2 we consider the smaller Zariski open subset $X\backslash \{f_{0,n-1,1}f_{0,n-1,2}=0\}$, which corresponds to considering the ideal $$\{g_1,\ldots ,g_m,f_{0,n-1,1}y_1+f_{0,n-1,2},1-t_{0,n-1}f_{0,n-1,1}f_{0,n-1,2}\},$$ then we can reduce the number of systems of equations regarding $\Gamma _{0,n-1}$ to $1$. Consequently, this reduces the number $C_1$ in the conclusion of Theorem \ref{TheoremMain}. Note, however, that in doing so we increase (double) the degree of the polynomials involved and consequently the number $C_3$.  
\end{proof} 
 
The following algorithm is a spin-off of the proof of Theorem \ref{TheoremMain}. 
 
{\bf Algorithm.} Detecting the existence of bounded birational maps between given varieties.  
\begin{itemize}
\item[] {\bf Input.} Two irreducible subvarieties $X$ and $Y$ of $\mathbb{P}^n$ of the same dimension, and a positive integer $n$.

\noindent {\bf Output.} An answer Yes or No to the question of whether there are birational maps $f:X\dashrightarrow Y$ which are restrictions of rational maps $F:\mathbb{P}^n\dashrightarrow \mathbb{P}^n$ whose degree $\deg (F)$ is $\leq d$. 

\item Step 1: Use a linear change of coordinates so that both $X$ and $Y$ do not belong to any coordinate hyperplane of $\mathbb{P}^n$. 

\item Step 2: Compute the explicit constant $C_1=C_1(d)>0$  (from part 2 of  Lemma \ref{LemmaDegreeGraph}) with the following property: Whenever $Z\subset \mathbb{C}^n\times \mathbb{C}^{n}$ is an irreducible variety of degree $\leq d$, and $\pi _j:\mathbb{C}^n\times \mathbb{C}^n\rightarrow \mathbb{C}^n\times \mathbb{C}^j$ is the natural linear projection (where $0\leq j\leq n$), then $\deg (\pi _j(Z))\leq C_1$. 

 \item Step 3: With respect to $C_1$ (considered as the degree of $Z$ in the statement of Lemma \ref{LemmaMonoidals}), compute the explicit constant $C_2$ (considered as the number $d$ in the statement of Lemma \ref{LemmaMonoidals}).  
 
 \item Step 4: Now we start to construct the systems of polynomial equations. We start with $S_{0,n}=\{g_1,\ldots ,g_m\}$, where $g_1,\ldots ,g_m$ is a basis for the ideal $I(X)$ defining $X$ in $\mathbb{C}^n$.
 
 \item Step 5: For each $j=1,2,\ldots ,n-1$, we construct the system $S_{0,n-j}$ as follows.
 
 The variables of $S_{0,n-j}$ are the union of the variables of $S_{0,n-j+1}$, the coefficients of $f_{0,n-j,1}$ and $f_{0,n-j,2}$, and an extra variable $t_{0,n-j}$. Here, $f_{0,n-j,1}$ is a general polynomial of degree $C_2-1$ in variables $x_1,\ldots ,x_n,y_1,\ldots ,y_{j-1}$, and $f_{0,n-j,2}$ is a general polynomial of degree $C_2$ in variables $x_1,\ldots ,x_n,y_1,\ldots ,y_{j-1}$. 
 
 The equations of $S_{0,n-j}$ are the union of the equations of $S_{0,n-j+1}$ and $2$ extra equations: $f_{0,n-j,1}y_j+f_{0,n-j,2}$ and $1-t_{0,n-j}f_{0,n-j,1}f_{0,n-j,2}$. (These two extra equations represent certain Zariski open dense sets of the involved monoids.)
 
\item Step 6: We now construct a system of polynomials $S_{0,0}$, which is more special than the $S_{0,j}$'s in Step 5. 

The variables of $S_{0,0}$ are the union of the variables of $S_{0,1}$, the coefficients of $f_{0,0,1}$, $g_{0,0}$, $h_{0,0}$ and $f_{0,0,2}$, and two extra variables $t_{0,0,1}$ and $t_{0,0,2}$. Here $f_{0,0,1}$, $g_{0,0}$, $h_{0,0}$ and $f_{0,0,2}$ are all general polynomials in unknowns $x_1,\ldots ,x_{n-1}, y_1,\ldots ,y_{n-1}$, of degrees $C_2$, $C_2-1$, $C_2-1$ and $C_2-2$ respectively.   

The equations of $S_{0,0}$ are the union of the equations of $S_{0,1}$ and $3$ extra equations: $f_{0,0,1}+x_ng_{0,0}+y_nh_{0,0}+x_ny_nf_{0,0,2}=0$, $1-t_{0,0,1}(f_{0,0,1}+x_ng_{0,0})(h_{0,0}+x_nf_{0,0,2})=0$ and $1-t_{0,0,2}(f_{0,0,1}+y_nh_{0,0})(g_{0,0}+y_nf_{0,0,2})=0$. 
 
 This $S_{0,0}$  - in case non-empty - represents certain non-empty Zariski open sets of graphs of rational maps from $X$ to $\mathbb{C}^n$. Each $S_{0,n-j}$ in Step 5 then represents the image of $S_{0,0}$ under the natural linear projections $\mathbb{C}^n\times \mathbb{C}^n\rightarrow \mathbb{C}^n\times \mathbb{C}^j$. All $S_{0,n-j}$ - when non-empty - are birational to $X$. 
 
 \item Step 7: We now continue to add more equations and variables. For each $j=1,2,\ldots ,n$ we construct the system $S_{j,0}$ as follows. 
 
 The variables of $S_{j,0}$ are the union of the variables of $S_{j-1,0}$, the coefficients of $f_{j,0,1}$ and $f_{j,0,2}$, two extra variables $t_{j,0}$ and $a$, together with coefficients of other polynomials $\tau _{j,1,0},\ldots ,\tau _{j,1,m(j)}$ and $\tau _{j,2,0}, \ldots ,\tau _{j,2,m(j)}$. Here $f_{j,0,1}$ is a general polynomial of degree $C_2-1$ in variables $y_1,\ldots ,y_n,x_{n-j-1},\ldots ,x_1$, and $f_{j,0,2}$ is a general polynomial of degree $C_2$. Here $m(j)=$ the number of polynomials in $S_{j-1,0}$, and $\tau _{j,1,0},\ldots ,\tau _{j,1,m(j)}$ and $\tau _{j,2,0}, \ldots ,\tau _{j,2,m(j)}$ are general polynomials in the variables of $S_{j,0}$ and $a$. The degrees of $\tau _{j,1,0},\ldots ,\tau _{j,1,m(j)}$ and $\tau _{j,2,0}, \ldots ,\tau _{j,2,m(j)}$ are determined from the effective Hilbert's Nullstellensatz, related to the two extra equations in the next paragraph. 
 
 The equations of $S_{j,0}$ are the union of the polynomials $h_1,\ldots ,h_{m(j)}$ of $S_{j-1,0}$ and $2$ extra equations: 
 \begin{eqnarray*}
 &&-1+\tau _{j,1,0}(1-a(f_{j,0,1}x_j+f_{j,0,2}))+\tau _{j,1,1}h_1+\ldots \tau _{j,1,m(j)}h_{m(j)}=0,\\
 &&-1+\tau _{j,2,0}(1-a(1-t_{j,0}f_{j,0,1}f_{j,0,2}))+\tau _{j,2,1}h_1+\ldots \tau _{j,2,m(j)}h_{m(j)}=0.
 \end{eqnarray*} 
 
 The above 2 equations check that $\{h_1=\ldots =h_m=0\}$ belongs to both the monoid $(f_{j,0,1}x_j+f_{j,0,2})=0$ and the set $f_{j,0,1}f_{j,0,2}\not= 0$ where the projection from the monoid has good properties.  Recall from the classical fact used in the proof of Theorem \ref{TheoremMain} that to check whether $\{h_1=\ldots =h_m=0\}$ belongs to the set $\{h=0\}$, we add a new variable $a$ and use the effective Hilbert Nullstellensatz to reduce the question to the existence of polynomials $\tau ,\tau _1,\ldots ,\tau _m$ (in the variables of the polynomials $h_1,\ldots ,h_m,h$ and the new variable $a$) of effectively bounded degrees so that $1=\tau (1-ah)+\tau _1 h_1+\ldots +\tau _mh_m$. 
 
 At the end of Step 7, we obtain a system of polynomial equations $S_{n,0}$ - which, in case non-empty - represents certain non-empty Zariski open subset of the images of $X$ under birational maps. 
 
\item Step 8: We are now ready to construct the final system of polynomial equations $S$. 

The variables of $S$ are the union of the variables of $S_{n,0}$,  the coefficients of $F_0,F_1,\ldots ,F_n$, an extra variable $a$, together with the coefficients of polynomials $\tau _{i,j,k}$ coming from effective Hilbert Nullstellensatz (more detail below). Here $F_0,F_1,\ldots ,F_n$ are general polynomials of degree $\leq d$ in variables $x_1,\ldots ,x_n$. 

The equations of $S$ are the union of equations in $S_{n,0}$ and the ones coming from effective Hilbert Nullstellensatz when we want to check that the set defined by $S_{n,0}$ is contained in the set $\{y_1F_0-F_1=0,\ldots, y_nF_0-F_n=0, 1-aF_0=0\}$ and $Y$. The number of  these extra equations equals the sum of $n+1$ (which is the number of polynomials in $\{y_1F_0-F_1=0,\ldots, y_nF_0-F_n=0, 1-aF_0=0\}$) and the number of a basis of the ideal defining $Y$. Each of these polynomials will contain the set $S_{n,0}$ and leads to one equation as we mentioned in Step 7. 

This $S$ represents the fact that the birational maps induced by solutions of $S_{n,0}$ are restrictions of rational selfmaps of $\mathbb{P}^n$ degrees $\leq d$, and maps $X$ onto $Y$. 

\item Step 9: Now we can use Gr\"oebner bases to solve the system of polynomial equations $S$. If $S$ has at least one solution, then the output of the algorithm is Yes, that is there is a birational map of degree $\leq d$ from $X$ onto $Y$. If, on the contrary, $S$ has no solution, then the output of the algorithm is No.  
\end{itemize}

Now we give the proof of Corollary \ref{CorollaryMain}. 
\begin{proof}[Proof of Corollary \ref{CorollaryMain}]
That $\mathcal{R}(n,d)$ is an algebraic variety is easy to see. The last paragraph of Step 12 in the proof of Theorem \ref{TheoremMain} shows that $\mathcal{R}^+(n,d)$ is an algebraic variety and can be explicitly constructed. 

The proof of Theorem \ref{TheoremMain} shows that $W(X,Y,d)$ is a subvariety of $\mathbb{C}^N\times \mathcal{R}(n,d)$ for some integer $N$, and the map $\kappa :W(X,Y,d)\rightarrow \mathcal{B}(X,Y,d)$ is simply the restriction to $W(X,Y,d)$ of the projection $\mathbb{C}^N\times \mathcal{R}(n,d)\rightarrow \mathcal{R}(n,d)$. Therefore, by Chevalley's theorem, the subset $\mathcal{B}(X,Y,d)$ of $\mathcal{R}(n,d)$ is constructible. The fact that a constructible set is an algebraic variety is clear. In fact, such a set is a finite union of sets of the form $A\cap B$, where $A$ is a closed subvariety of $\mathcal{R}(n,d)$ and $B$ is the complement of a hypersurface. By adding a new variable, as in the proof of Theorem \ref{TheoremMain}, we can write explicitly the ideal defining $B$. 

The proof for the subset $\mathcal{B}^+(X,Y,d)$ of $\mathcal{R}^+(n,d)$ is similar. 
\end{proof}

\begin{remark} We list some final remarks in this section. \end{remark}

\begin{itemize}
\item  For a {\bf given} rational map $f:X\dashrightarrow Y$, it has been known for some decades that Gr\"obner bases can be used to check whether it is a birational map (see e.g. \cite{shannon-sweedler} for the special case when $Y$ is a projective space). However, as far as we know, our algorithm in the proof of Theorem \ref{TheoremMain} is the first one to treat the case of determining the set of all birational maps from $X$ to $Y$, which need to deal with maps whose coefficients are {\bf undetermined}. This poses difficulties which do not arise in the case of explicit maps. 

\item While the bounds in the proof of Theorem \ref{TheoremMain} are explicit, they are quite big. For example, the degree bound in the effective Hilbert Nullstellensatz is exponential, and the degree bound for the graph $\Gamma _f$ is polynomial. Therefore, more refinements of the given algorithm are needed before it can be used for practical examples. Also, with the ever improvement in computer hardwares and softwares, we hope that in a not far future the theoretical aspects of results in this paper can be put into practice. 

\item As the Addendum in the proof of Theorem \ref{TheoremMain} shows, there are many choices of such $W(X,Y,d)$ and $\kappa$. We know that whatever the choice of $W(X,Y,d)$ is, its image (as a set) by $\kappa$ is always $\mathcal{B}(X,Y,d)$. By Corollay \ref{CorollaryMain}, $\mathcal{B}(X,Y,d)$ is an algebraic variety. Hence, we can choose $\mathcal{B}(X,Y,d)$ as a canonical choice for all such $W(X,Y,d)$'s. Using the recent result \cite{harris-michalek--sertoz} on computing images of polynomial maps, we would again be able to explicitly describe $\mathcal{B}(X,Y,d)$ in terms of $X,Y$ and $d$. The same consideration can also be applied to $\mathcal{B}^+(X,Y,d)$.  

If we are interested only in finding an answer to the Bounded Birational Problem, we may alternatively proceed in the following simpler manner. We start with any of the $W(X,Y,d)$ given in Theorem \ref{TheoremMain}. Then elimination (by using for example Gr\"obner bases) gives us explicitly a closed subvariety $\mathcal{C}(X,Y,d)$ of $\mathcal{R}(n,d)$. While $\mathcal{C}(X,Y,d)$ may not be reduced, its reduced structure is exactly $\overline{\mathcal{B}(X,Y,d)}$. Again, the reduced structure of $\mathcal{C}(X,Y,d)$ can be explicitly constructed using Gr\"obner bases. In general, this $\overline{\mathcal{B}(X,Y,d)}$ may be bigger than $\mathcal{B}(X,Y,d)$, and hence a specific point in $\overline{\mathcal{B}(X,Y,d)}$ may not give rise to a birational map from $X$ to $Y$. However, still we have that $\mathcal{B}(X,Y,d)\not= \emptyset$ if and only if $\overline{\mathcal{B}(X,Y,d)}\not= \emptyset$. Therefore, $\overline{\mathcal{B}(X,Y,d)}$ can still be used to answer the Bounded Birationality Problem, and from the above analysis we may regard $\overline{\mathcal{B}(X,Y,d)}$ as another {canonical choice} concerning this aspect. 
\end{itemize}

\section{A rough strategy towards the birationality problem via Iitaka's fibrations}
In this section, we state a rough strategy, based on Theorem \ref{TheoremMain}, for a computational approach towards the birationality problem.  

We first recall briefly about the fact that the birationality problem for varieties of general type over $\mathbb{C}$ was completely solved about 10 years ago. We thank John Christian Ottem for communicating this remark. Note that the birationality problem for smooth surfaces of general type is {\bf computable}. For any smooth surface of general type $X$, the linear system $|5K_X|$ is a birational embedding of $X$ into some projective space. Given two smooth projective surfaces of general type $X,Y$, if $h^0(X,5K_X)\not= h^0(Y,5K_Y)$, then they are not birationally equivalent. In the case $h^0(X,5K_X)= h^0(Y,5K_Y)=N$, let $X'$ be the birational embedding of $X$ in $\mathbb{P}^N$ using the linear system $|5K_X|$ and $Y'$ be the birational embedding of $Y$ in $\mathbb{P}^N$ using the linear system $|5K_Y|$. If $f:X\dashrightarrow Y$ is birational, then $f^*:H^0(X,5K_X)\rightarrow H^0(Y,5K_Y)$ is isomorphic, and hence there is a linear map in $\mathbb{P}^N$ which maps $X'$ onto $Y'$. The latter question is computable. Similarly, the question of whether two smooth complex projective varieties of {\bf general type} in higher dimension are birational equivalent is {\bf decidable}. This follows from the following result in \cite{hacon-mckernan, takayama, tsuji1, tsuji2}: Let $X$ be a smooth irreducible projective variety of dimension $k$. Then there is a positive number $r_k>0$ depending only on $k$ such that the linear system $|r_kK_X|$ is a birational embedding. If $k>2$, it is not known whether this question is again {\bf computable}, since the above number $r_k$ is not yet explicitly determined. 

Note that in the proof of Theorem \ref{TheoremMain}, the assumption that the birational map $f:X\dashrightarrow Y$ is the restriction of a rational map $F:X\dashrightarrow Y$ of degree $\leq d$ is needed only to deduce that the degree of the graph $\Gamma _g$ is explicitly bounded. Therefore, the birationality problem is solved if the following question has an affirmative answer.
 
\begin{question}[Main Question] Given $X,Y\subset \mathbb{P}^n$ be irreducible varieties, where $n\geq 2$, and assume that there is a birational map $f:X\dashrightarrow Y$. Is there another birational map $g:X\dashrightarrow Y$ so that the degree of the graph $\Gamma _g$, viewed as a subvariety of $\mathbb{P}^n\times \mathbb{P}^n$, is explicitly bounded in terms of $n$ and the ideals for $X$ and $Y$?
\label{Question1}\end{question} 
 
A special case when Question \ref{Question1} is answered in the affirmative is when $X$ is a monoid and $Y$ is a projective space, see the description before the proof of Theorem \ref{TheoremMain}. The bound in degree is the degree of the monoid. Less trivially, we have the following result. 
 
\begin{proposition}

1) To solve Question \ref{Question1}, it is sufficient to do it for the case where $X$ and $Y$ are either: 

a) Hypersurfaces;

or

b) Normal varieties.

2) Over $\mathbb{C}$, Question \ref{Question1} has an affirmative answer when $X$ and $Y$ are varieties of general type.

3) Question \ref{Question1} has an affirmative answer when $X$ and $Y$ are curves. 
\label{PropositionMain}\end{proposition}
\begin{proof}

1)
a) Let $m=\dim (X)+1$, we can always find a linear projection $\pi :\mathbb{P}^n\dashrightarrow \mathbb{P}^m$, under which the strict transform $X'$ of $X$ is birational to $X$ and the strict transform $Y'$ of $Y$ is birational to $Y$. Assume that Question \ref{Question1} is affirmatively answered for hypersurfaces. Then, provided there is a birational map from $X$ to $Y$, there will be a birational map $g':X'\dashrightarrow Y'$ whose graph $\Gamma _{g'}\subset X'\times Y'\subset \mathbb{P}^m\times \mathbb{P}^m$ has explicitly bounded degree in terms of $X,Y$. Then the graph $\Gamma _g\subset X\times Y\subset \mathbb{P}^n\times \mathbb{P}^n$ of the lifting birational map $g:X\dashrightarrow Y$, which is contained in the intersection between $X\times Y$ and the strict transform of $\Gamma _{g'}$ under the dominant rational map $\pi \times \pi$, also has explicitly bounded degree. Then the proof of Theorem \ref{TheoremMain} provides us with a rational map $F:\mathbb{P}^n\dashrightarrow \mathbb{P}^n$, whose degree is explicitly bounded in terms of $X,Y$, and whose  restriction to $X$ is $g$.  

b) The proof is similar, now we use that any variety has a normalisation. 

2) By the results mentioned above, there is a number $r_k>0$ depending only on $k=\dim (X)=\dim (Y)$ so that the pluricanonical divisor $r_kK_X$ gives a birational map $X\dashrightarrow M(X)\subset \mathbb{P}^{N_1}$, and similarly $r_kK_Y$ gives a birational map $Y\dashrightarrow M(Y)\subset \mathbb{P}^{N_2}$. From the properties of the canonical divisor, we have that $X$ and $Y$ are birational iff $N_1=N_2=:N$ and $M(X)$ is isomorphic to $M(Y)$ via a linear automorphism of $\mathbb{P}^N$. Then, the degree of the graph of the composition $X\dashrightarrow M(X)\rightarrow M(Y)\dashrightarrow Y$ is bounded effectively in terms of $X$ and $Y$.  

3) We consider different cases. 

Case 1: $X$ and $Y$ are curves of genus $\geq 2$. We can apply the argument in 2). 

Case 2: $X$ and $Y$ are elliptic curves.  We use that on an elliptic curve the divisor $3p$, where $p$ is any point in the elliptic curve gives rise to an embedding of that elliptic curve as a cubic curve in $\mathbb{P}^2$. Now, two smooth cubic curves in $\mathbb{P}^2$ are isomorphic iff they are so under some linear automorphisms of $\mathbb{P}^2$. Then, the same argument as in 2) completes the proof. 

Case 3: $X$ and $Y$ are rational curves. In this case we can use the anti-canonical divisor and argue as in 2). 
\end{proof}

Proposition \ref{PropositionMain} gives some evidence that Question \ref{Question1} should have an affirmative answer. In the remaining of this section, we provide a rough argument, based on Iitaka's fibration, to reduce Question \ref{Question1} to the following two special cases: varieties of Kodaira dimension $-\infty$ or $0$. The argument we present here is largely heuristic and relies on further advancement on the understanding of Iitaka's fibrations. 

In the remaining of this section we work over $\mathbb{C}$. Iitaka has shown that for a smooth projective variety $Z$ of non-negative Kodaira dimension, for $m$ large enough so that $h^0(Z,mK_Z)\not= 0$, then the maps $Z\dashrightarrow \mathbb{P}H^0(Z,mK_Z)^* $ are all birational to one common fibration, now so-called the Iitaka's fibration. Moreover, the Iitaka's fibration has a universal property among fibrations whose general fibers have zero Kodaira's dimension, which is unique up to birational morphisms. Under some assumptions (in particular, including that the general fiber of the Iitaka's fibration has good minimal models), \cite{pacienza} showed that the bound on $m$ is effectively computed in terms of the dimension of $Z$, the pluricanonical divisors of general fibers, and the Betti numbers of some associated cover over the general fibers. More recently, \cite{birkar-zhang} (Theorem 1.2 therein) proved the same result unconditionally.

Now let $X$ and $Y$ be two smooth projective varieties of the same Kodaira dimension $\geq 0$. By the cited result, there is an effective constant $m$ (in terms of $X$ and $Y$), so that both natural maps $\varphi _X:X\dashrightarrow IF(X)\subset \mathbb{P}H^0(X,mK_X)^* $ and $\varphi _Y:Y\dashrightarrow IF(Y)\subset \mathbb{P}H^0(Y,mK_Y)^*$ are both birational to the corresponding Iitaka's fibrations of $X$ and $Y$. Note that $IF(X)$ and $IF(Y)$ are varieties of general type of the same dimension. By the uniqueness of Iitaka's fibration upto birational maps, and the properties of the pluricanonical divisors mentioned in the proof of part 2 of Proposition \ref{PropositionMain},  we deduce that $X$ and $Y$ are birational iff there is a birational map $\psi :X\dashrightarrow Y$ which takes a general fiber of $\varphi _X$ to a general fiber of $Y$, and that the induced map $IF(X)\dashrightarrow IF(Y)$ is the restriction of a linear automorphism of $\mathbb{P}H^0(X,mK_X)^*=\mathbb{P}H^0(Y,mK_Y)^*$. Note then that also the degrees of the graphs of the birational maps between the general fibers of $\varphi _X$ and $\varphi _Y$  are uniformly bounded in terms of the degrees of intersection between the graph of $\psi$ and the products of the general fibers of the concerned Iitaka's fibrations. 

We now assume that Question \ref{Question1} has an affirmative answer for varieties of Kodaira dimensions $0$ and $-\infty$. Then we will next argue heuristically that Question \ref{Question1} also has an affirmative answer in the general case. In fact, given $X$ and $Y$ two varieties of the same Kodaira dimension $\geq 0$. Let $m$ be the effective constant in \cite{birkar-zhang}, and $\varphi _X:X\dashrightarrow IF(X)$ and $\varphi :Y\dashrightarrow IF(Y)$ be the corresponding Iitaka's fibrations. If $\mathbb{P}H^0(X,mK_X)^*$ and $\mathbb{P}H^0(Y,mK_Y)^*$ are not isomorphic, then we conclude that $X$ and $Y$ are not birational, and nothing else needs to be done. So, we can assume that $\mathbb{P}H^0(X,mK_X)^*=\mathbb{P}H^0(Y,mK_Y)^*$. Then we can construct an algebraic variety parametrising all isomorphisms from $IF(X)$ to $IF(Y)$ which are restrictions of linear automorphisms of $\mathbb{P}H^0(X,mK_X)^*=\mathbb{P}H^0(Y,mK_Y)^*$. Now, using the assumption that Question \ref{Question1} has an affirmative answer for varieties of Kodaira dimension $0$, we construct for each general fiber of $\varphi _X$ an algebraic variety representing all birational maps onto the corresponding fiber of $\varphi _Y$ (under the isomorphism between $IF(X)$ and $IF(Y)$ mentioned in the previous sentence) whose graphs have degrees bounded effectively in terms of the specific fibers. Now, since the fibers vary algebraically, we expect that there will be a uniform bound on the degrees of the graphs mentioned in the last sentence. Then we expect that these seperate varieties, one for each general fiber of $\varphi _X$, can be put together to give a variety parametrising birational maps (of a special form, preserving the given Iitaka's fibrations) from $X$ to $Y$. Combining all the above construction, we see at the same time that Question \ref{Question1} has an affirmative answer.

 So, assuming that the heuristic argument in the above paragraph works (which will, as can be easily seen, require a deeper understanding of Iitaka's fibrations), for solving the birationality problem it remains to solve Question \ref{Question1} for varieties of Kodaira dimensions $0$ and $-\infty$. For varieties of Kodaira dimension $0$, we speculate that some generalisation of the argument for elliptic curves in the proof of part 3) of Proposition \ref{PropositionMain} may be useful. For varieties of Kodaira dimension $-\infty$, it may be useful to look at the pluri-anticanonical divisors $-mK_X$.     

 \section{Variants} In this section we prove similar results for other interesting classes of maps (such as dominant rational maps, regular morphisms, isomorphisms or regular embeddings). The proofs will combine the ideas in the previous sections together with some additional ingredients to resolve new difficulties associated with the particular case at hand.  As before, we work with an algebraically closed field $\mathbb{K}$ of arbitrary characteristic, unless specifically stated otherwise. We fix from now on two irreducible subvarieties $X,Y\subset \mathbb{P}^n$, and a positive integer $d$. 
 
We also show that some minor modifications prove the same results for maps on affine varieties. Then we deduce from this the validity of all the results for maps on arbitrary algebraic varieties.  
 
 \subsection{Dominant rational maps of bounded degrees} We consider the following question. 
 \begin{question}[Question (A)] Is there a rational map $F:\mathbb{P}^n\dashrightarrow \mathbb{P}^n$, of degree $\leq d$, so that $F|_X$ is a dominant rational map onto $Y$?
\end{question} 
 In the case where $X=\mathbb{P}^k$ a projective space, we have the unirationality problem, which has attracted a lot of interest. 
 \begin{theorem}
 The above Question (A) is computable. 
 \label{Theorem1}\end{theorem} 
 \begin{proof}
We will give an algorithm, whose complexity is explicitly bounded, to solve this question. 

We write the maps $F:\mathbb{P}^n\dashrightarrow \mathbb{P}^n$ needed to find in terms of their coefficients.   

First, to check that $X$ is not contained in the indeterminacy locus $F$ and that $F$ maps $X$ to $Y$, we can proceed as in the proof of Theorem \ref{TheoremMain}. This way we get some systems of polynomial equations. 

We now proceed to check the condition that $F$ maps $X$ onto $Y$. To this end, we analyse what happens in the opposite case. So, assume for a moment that the map $F$ does not map $X$ onto $Y$. This means that there is a proper subvariety $W$ of $Y$ so that the projection of $\Gamma _f$ is $W$. We observe that, by part 2 of Lemma \ref{LemmaDegreeGraph}, the degree of $W$ is bounded in terms of that of $\Gamma _f$. Then since $\mathbb{P}^n\times \mathbb{P}^n$ can be embedded into $\mathbb{P}^{n^2+2n}$, we can use a trick of Mumford \cite{mumford} that {\bf set-theoretically} $W$ is generated by polynomials of degrees explicitly bounded in the degree of $W$ (in fact, these can be chosen as cones over $W$ with vertex a linear subspace of $\mathbb{P}^n$, and hence of degree $=$ $\deg (W)$), to find a proper hypersurface $Z$ of explicitly bounded degree of $\mathbb{P}^n$ so that $Y\not\subset Z$, $Z\cap Y$ contains $\pi (\Gamma _f)$. We now work on the affine Zariski open set $\mathbb{C}^n\times \mathbb{C}^n$ of $\mathbb{P}^n\times \mathbb{P}^n$, and interpret the previous sentence in terms of polynomial equations.  Since the degree of $Z$  is explicitly bounded, we can as before parameterise it in terms of an explicitly bounded number of coefficients. As in the proof of Theorem \ref{TheoremMain}, we can express the condition that $Z\cap Y$ contains $\pi (\Gamma _f)$ which we call (E). 

Now we let $(E')$ to be the union of $(E)$ and the polynomial equations expressing the condition that the hypersurfaces $Z$ in the above paragraph contain $Y$. Then the variety defined by $(E')$ is a subvariety of the variety defined by $(E)$. The maps $F$ which maps $X$ onto $Y$ will be parametrised by the complement of $(E')$ in $(E)$, and hence are parametrised by an algebraic variety. Hence to check that there are dominant rational maps $F$ from $X$ onto $Y$ is the same as checking that the set of solutions to $(E)$ is strictly bigger than the set of solutions to $(E')$, which is a decidable problem, by using for example Gr\"obner bases.   
 
\end{proof}
 
 \subsection{Regular morphisms of bounded degrees}  We now consider the following question. 
 \begin{question}[Question (B)] Is there a rational map $F:\mathbb{P}^n\dashrightarrow \mathbb{P}^n$, of degree $\leq d$, so that $F|_X$ is a regular morphism into $Y$?
 \end{question}
 
 Note that while the indeterminacy set of $F=[F_0:F_1:\ldots :F_n]$, as a rational map from $\mathbb{P}^n\dashrightarrow \mathbb{P}^n$, is simply the set $\mathcal{I}(F):=\{F_0=\ldots =F_n=0\}$, determining the indeterminacy set of the restriction of $F$ to a subvariety $X$ is not an easy task. In particular, it is not true that the indeterminacy set of $F|_X$ is always $\mathcal{I}(F)\cap X$. A classical example is the following. Let $X=\{xz=y^2\}\subset \mathbb{P}^2$. The rational map $F:\mathbb{P}^2\dashrightarrow \mathbb{P}^1$ given by $[x:y:z]\mapsto [x:y]$ has $\mathcal{I}(F)=[0:0:1]\in X$. However, $F|_X$ is a regular morphism. In fact, this follows from the fact that we have $[x:y]=[y:z]$ on $X$, and at least one among these two represents a genuine point in $\mathbb{P}^1$. One can also check that $[x:y]$ and $[y:z]$ are the only representatives (upto multiplicative factors) of $F|_X$, and the indeterminacy set of $[y:z]$ is $[1:0:0]\in X$. Thus for any $G:\mathbb{P}^2\dashrightarrow \mathbb{P}^1$ with $G|_X=F|_X$ we have that $\mathcal{I}(G)\cap X\not=\emptyset$.     
 
For a given rational map $F$, \cite{simis} computed the indeterminacy set of $F|_X$ in terms of some algebras associated to $F$. Hence by checking that these algebras give rise to empty indeterminacy sets, we arrive at a necessary and sufficient condition for $F|_X$ to be regular. However, as far as we know, no criterion has been given in the literature for the case where the map $F$ is not explicitly given. The main result we give here is that provided $X$ is a smooth projective variety,  then Question (B) is computable. 

\begin{example} Before stating the main result, let us explain the main idea via the example $X=\{xz=y^2\}\subset \mathbb{P}^2$ and $F[x:y:z]=[x:y]$ above. Let $\pi :\mathbb{P}^2\times \mathbb{P}^1\rightarrow \mathbb{P}^2$  be the projection to the first factor. We saw that for $f=F|_X$, the graph $\Gamma _f$ is not the intersection between $\Gamma _F$ and $\pi ^{-1}(X)$ (which is $\{([x:y:z],[u:v])\in \mathbb{P}^2\times \mathbb{P}^1:~xz=y^2,~xv=yu\}$), but it is an irreducible component of this intersection. Therefore, we can add in several polynomials to obtain the correct ideal for $\Gamma _f$. In this case, it turns out that we need only to add one more polynomial, of degree $2$. More precisely 
$$\Gamma _f=\{([x:y:z],[u:v])\in \mathbb{P}^2\times \mathbb{P}^1:~xz=y^2,~xv=yu,~yv=zu\}.$$
To check that $f$ is a regular morphism is the same as checking that the projection $\pi $ maps $\Gamma _f$ isomorphically to $X$. Since $\pi$ is a birational morphism on $\Gamma _f$, it suffices to check that the {\bf relative tangent spaces} (defined by taking with respect to variables in $\mathbb{P}^1$ of the defining equations for $\Gamma _f$) are $0$ at every point of $\Gamma _f$. (In fact, we can work on a chart, say $\mathbb{C}_{x,y}^2\times \mathbb{C}_u$. If $f$ is regular, then by Hilbert's Nullstellensatz in this chart $\Gamma _f$ is generated by $u-g(x,y)=0$ and the ideal for $X$, which when taking derivative with respect to $u$ will have rank $1$ as claimed. Then if $g_1,\ldots ,g_m$ are another set of generators for the ideal of $\Gamma _f$ in this chart, we will get the same answer, which is $1$. Hence the dimension of the relative tangent space, which is the corank of the above matrix, is $0$. Conversely, if the relative tangent spaces are $0$ for every point on $\Gamma _f$,   this will give us that the projection $\pi$ is \'etale. Hence, by using Zariski's Main Theorem, since $\pi$ is also projective and birational, it is an isomorphism. Alternatively, we observe that any fibre of $\pi$ must be finite, and hence by cohomological reason has the same number of points, which is $1$.) This can be checked by working with the $(2+1)(1+1)=6$ coordinate charts $\mathbb{C}^2\times \mathbb{C}$ of $\mathbb{P}^2\times \mathbb{P}^1$. For example, in one such chart, corresponding with $z=1$ and $v=1$, we obtain
\begin{eqnarray*}
X&=&\{(x,y)\in \mathbb{C}^2:~x=y^2\},\\
\Gamma _f&=&\{((x,y),u)\in \mathbb{C}^2\times \mathbb{C}:~x=y^2,~x=yu,~y=u\}.
\end{eqnarray*}
The ideal of $\Gamma _f$ is generated by $x-y^2,x-yu,y-u$, and hence the dimension of the relative tangent space  is the corank of the matrix we obtain by taking the first derivatives of the above polynomials with respect to $u$: which consecutively are $0,-y,-1$. Since the rank of this matrix is $1$ everywhere on $\Gamma _f$, we conclude that the relative tangent space has dimension $0$ everywhere, as desired.
\label{Example1}\end{example}

\begin{remark}The point of view in Example \ref{Example1} is that to check whether $f$ is a regular morphism, we do not need to compute the indeterminacy set of $f$. In stead, we only need to check that $\Gamma _f$ is isomorphic to $X$, and this in turn is the same as checking that the projection $\Gamma _f\rightarrow X$ is \'etale. The special form of the projection makes computations less complicated and is more flexible, allowing us to deal with maps $F$ not explicitly given. 
\end{remark}

In Example \ref{Example1}, we see that the important factor for being able to arrive at an explicit algorithm is that the number of polynomials needed, as well as their degrees, to add into $\Gamma _F\cap \pi ^{-1}(X)$ - where $\pi :\mathbb{P}^n\times \mathbb{P}^n\rightarrow \mathbb{P}^n$ is the projection to the first factor - to obtain the ideal for $\Gamma _f$ (recall that $f=F|_X$), should be explicitly bounded in terms of $d$ and $X$. (Note that we will need to deal with the case where $f$ is undetermined, and so we do not know anything about $\Gamma _f$ except that it is contained in $\Gamma _F\cap \pi ^{-1}(X)$, its degree is explicitly bounded and it should be isomorphic to $X$. Hence if $X$ is smooth then $\Gamma _f$ is also smooth.) The following lemma (see Theorem 10 in \cite{blanco-jeronimo-solerno}) is important in this aspect. (For the convenience of the readers, we give a simple proof, based solely on \cite{mumford}, of a non-optimal version of the lemma, which is enough for the conclusions of Theorems \ref{Theorem2} and \ref{Theorem3}.) 
\begin{lemma}
Let $V\subset \mathbb{A}^n$ be a smooth equi-dimensional algebraic variety and set $m:=(n-\dim (V))(1+\dim (V))$. There exist polynomials $f_1,\ldots ,f_m$ in $\mathbb{K}[x_1,\ldots ,x_n]$ with degrees bounded by $\deg (V)$ such that $I(V)=(f_1,\ldots ,f_m)$.  
\label{Lemma2}\end{lemma}
\begin{proof} Here is a simple proof of the following weaker version of the lemma which is enough for applications in the remaining of this paper. 

Claim. There is a number $m$ depending only on $n$ and $\deg (V)$, so that there exist $m$ polynomials $f_1,\ldots ,f_m$ of degree $\leq \deg (V)$ for which $I(V)=(f_1,\ldots ,f_m)$. 

Proof of the claim. Since the variety $V$ is smooth, if $f_1,\ldots ,f_m$ are polynomials vanishing on $V$, so that $V$ is set-theoretically $\{f_1=\ldots =f_m=0\}$ and for which the derivatives $df_1,\ldots ,df_m$ generate $T_xV$ for all $x\in V$, then $I(V)=(f_1,\ldots ,f_m)$ (see \cite{mumford}).   

Let $r=\dim (V)\leq n$. Using the above observation, we choose a point $x_0\in V$ and $n-r$ polynomials $f_1,\ldots ,f_{n-r}$ of degrees $\leq \deg (V)$ vanishing on $V$ so that $df_1,\ldots ,df_{n-r}$ generate $T_{x_0}V$ (exist by \cite{mumford}). Then the subset $V_1 $ of $V$ where $df_1,\ldots ,df_{n-r}$ do not generate the tangent space is a strict subvariety defined by the vanishing of all the $(n-r)\times (n-r)$ minors of the matrix $df_1,\ldots ,df_{n-r}$. Hence by Lemma \ref{Lemma0}, the number of irreducible components of $V_1$, as well as their degrees, are explicitly bounded. Moreover, $\dim (V_1)\leq \dim (V)-1$. We then choose, for each irreducible component of $V_1$, a generic point. For each of these points, we add in $\leq n-r$ polynomials of degrees $\leq \deg (V)$ so that to generate the tangent spaces at these points. Then the subset $V_2$ of $V$ where all the polynomials we already constructed do not generate tangent spaces is again a subvariety for which we can bound the number of irreducible components and their degrees. Moreover, $\dim (V_2)\leq \dim (V)-2$. We can use the same procedure as before, and after doing this at most $r$ times, the polynomials we construct will generate tangent spaces at every point of $V$. Note then that the variety which is the reduced structure of the ideal $\{f_1,\ldots ,f_q\}$ at the end is a disjoint union of $V$ and some other irreducible components. We can add in some more polynomials $g_1,\ldots , g_{p}$ of degree $\leq \deg(V)$ (by avoiding generic points of the irreducible components besides $V$ and proceed by induction as before) so that the zero set of $\{f_1=\ldots =f_q=g_1=\ldots =g_p=0\}$ is exactly $V$. Hence, they generate the ideal of $V$ as desired. 
\end{proof}

Now we can state the main result of this subsection.  
\begin{theorem}
Assume that $X$ is a smooth projective variety. Then Question (B) is computable. 
\label{Theorem2}\end{theorem}
 \begin{proof}
 Again, checking that a rational map $F:\mathbb{P}^n\dashrightarrow \mathbb{P}^n$ gives rise to a rational map from $X$ into $Y$ can be expressed in terms of polynomial equations. (This does not require any condition on $X$.) 
 
 We now proceed to showing that when $X$ is a smooth projective variety, the conditions for that $f=F|_X$ is regular, as a map from $X$ into $\mathbb{P}^n$, can also be expressed in terms of polynomial equations.  
 
As seen from Example \ref{Example1}, that $f$ is regular is equivalent to $\Gamma _f$ is isomorphic to $X$ and hence in particular is smooth. By Lemma \ref{Lemma2}, we can add an explicitly bounded number of polynomials of explicitly bounded degrees and which vanish on $H_{0,0}$ in the proof of Theorem \ref{TheoremMain} (this conditon can be expressed in terms of polynomial equations by the argument in the proof of Theorem \ref{TheoremMain}) into the polynomials already generating $\Gamma _F\cap \pi ^{-1}(X)$, where $\pi :\mathbb{P}^n\times \mathbb{P}^n\rightarrow \mathbb{P}^n$ is the projection to the first factor, to obtain generators $h_1,\ldots ,h_p$ for some $\Gamma$ (which we want to be the graph $\Gamma _f$). Note that the number $p$ and the degrees of $h_1,\ldots ,h_p$ are explicitly bounded.   

We want that $\Gamma$ is exactly $\Gamma _f$ and moreover it is isomorphic to $X$. Note that these two conditions are equivalent to that there is one such constructed $\Gamma$ which is  isomorphic to $X$. In fact, if this is the case, then $\Gamma $ is in particular irreducible. Since $\Gamma$ contains $H_{0,0}$ as a {\bf set}, it follows from the fact that $\Gamma$ is isomorphic to $X$, that $\Gamma$ is exactly the graph $\Gamma _f$. 

Now to check that $\Gamma$ is isomorphic to $X$, we need to check that the relative tangent spaces at every point in $\Gamma$ has dimension $0$.  (Strictly speaking, this check only shows that the irreducible component of $\Gamma$ containing $\Gamma _f$ as a  set is isomorphic to $X$, but this is enough to have that $\Gamma _f$ is isomorphic to $X$ as desired. A priori, since we are not given explicitly the map $F$ and the polynomials we add in are also undetermined, there may be other components of $\Gamma$ besides $\Gamma _f$.) We can work with Zariski open sets $\mathbb{C}^n\times \mathbb{C}^n$ of $\mathbb{P}^n\times \mathbb{P}^n$. From the above generators $h_1,\ldots ,h_p$ of $\Gamma$, which are undetermined, hence need to be parametrised in terms of their coefficients, we compute the derivatives of them with respect to the variables $y_1,\ldots ,y_n$ in the second factor of $\mathbb{C}^n\times \mathbb{C}^n$. Call the Jacobian matrix obtained $J_{\Gamma /X}$, and call $M_1,\ldots ,M_q$ all the $n\times n$-minors of $J_{\Gamma /X}$. Note that the number $q$ and the degrees of $M_1,\ldots ,M_q$ are also explicitly bounded. Then the fact that the relative tangent spaces at every point  of $\Gamma$ has dimension $0$ is translated into that the system defined by $M_1=\ldots =M_q=0$ has no solution on $\Gamma$. As in the proof of Theorem \ref{TheoremMain}, this and the effective Hilbert's Nullstellensatz allow us to translate Question B into the question about the existence of solutions to an explicitly constructed variety.  

 \end{proof}   
  
 \subsection{Isomorphisms/regular embeddings of bounded degrees} We now consider the following question. Question (C): Is there a rational map $F:\mathbb{P}^n\dashrightarrow \mathbb{P}^n$, of degree $\leq d$, so that $F|_X$ is an isomorphism onto $Y$? We also consider a more general question. Question (C'): Is there a rational map $F:\mathbb{P}^n\dashrightarrow \mathbb{P}^n$, of degree $\leq d$, so that $F|_X$ is a regular embedding into $Y$? 

Combining the previous results, we obtain the following result. 
\begin{theorem}
Assume that $X$ is a smooth projective variety. Then Questions (C) and (C') are computable.
\label{Theorem3}\end{theorem}
\begin{proof}
We only need to prove for Question (C').

Using Theorem \ref{TheoremMain}, we can check whether there is a rational map $F:\mathbb{P}^n\dashrightarrow \mathbb{P}^n$ of degree $\leq d$ so that $f=F|_X$ is a birational map into $Y$. Call $Z$ the image of $f$. (Here, we do not need to specify $Z$.) 

We can use Theorem \ref{Theorem2} to check if $F|_X$ is a regular morphism.  

It remains to check whether the birational map $g^{-1}:Z\dashrightarrow X$,  where $g=f^{-1}$, is also a regular morphism. To this end, we can compute the derivatives of the generators we have for $\Gamma _f$ in the proof of Theorem \ref{Theorem2} with respect to variables in the copy of $\mathbb{P}^n$ containing $Y$, and denote the resulting matrix by $J_{\Gamma /Z}$. As in the proof of Theorem \ref{Theorem2}, we only need to check that $J_{\Gamma /Z}$ gives rise to that the relative tangent spaces with respect to the projection $\mathbb{P}^n\times \mathbb{P}^n\rightarrow \mathbb{P}^n$ to the second factor are $0$. As before, this can be described in terms of polynomial equations involving minors of $J_{\Gamma /Z}$. 

(Note that here a priori, we do not know whether $Z$ is smooth or not. To this end, we can add the equations for the condition that $Z$ is smooth: it is the same that at least one of the $(n-\dim (X))\times (n-\dim (X))$ minors of the derivatives of the generators of $I(Z)$ must be non-zero at every point of $Z$, which then can be interpreted by effective Hilbert's Nullstellensatz.) 

At the end, we obtain some explicit systems of polynomial equations, whose solutions parametrise Question (C). 
\end{proof}

\subsection{Maps on affine varieties}\label{SectionAffine} In this subsection we discuss the above results for the cases where varieties concerned are affine. More precisely, considering now $X,Y\subset \mathbb{C}^n$ be closed irreducible varieties, and a positive integer $d$. We ask whether there is a rational map $F:\mathbb{C}^n\dashrightarrow \mathbb{C}^n$ of degree $\leq d$ so that the restriction $f=F|_X$ maps $X$ into $Y$ and : i) birational; or ii) dominant rational; or iii) regular; or iv) regular embedding. We also ask the following stronger versions of iii) and iv), which are not necessary in the projective setting: whether $f$ is iii') regular and surjective; or iv') a regular embedding onto a closed subvariety of $Y$. We recall that the reason that in the affine setting iii') and iv') are actually stronger than iii) and iv) respectively is because the image of a closed subvariety of $\mathbb{C}^n$ under a polynomial map is not always a closed subvariety but only constructible (Chevalley's theorem). In the above formulations, we may as well ask for the case where $F$ is a polynomial, in which case the proof will be the same.     

We will show that all of the above questions are computable. As usual, the degree of an affine variety is the degree of its closure in the projective space. With this convenience, the proofs of the previous results apply straightly forwardly to questions i) and ii), and show that they are computable. For question iii), the application is almost straightly forward, after we projectivise the varieties and the maps. Then an application of Zariski's main theorem (which says that the preimage of a smooth point in a variety $V$ by a birational morphism is connected, and hence is one point if the fibre has an isolated point) will again give that the fact that relative tangent spaces are $0$ at every points on the graph gives that the projection $\pi$ onto the first factor is a regular embedding of the graph $\Gamma _f$ into $X$ and vice versa. To check that the graph $\Gamma _f$ is the graph of a regular morphism, we need to check that the projection $\pi :\Gamma _f\rightarrow X$ is surjective. This then can be checked as in the end of the proof of iv') below. Since constructible sets form a Boolean algebra (close under finite union and complementation), it follows that we can parametrise question iii) also by algebraic varieties. For question iv), we use iii), and then similarly check that the relative tangent spaces of the projection $\Gamma _f\rightarrow Y$ are all $0$.  

Now we briefly show how to solve questions iii') and iv'). First we consider question iii'). By question iii), we can check whether $f$ is regular and whose image is contained in $Y$. To check that the image is actually $Y$ and hence solve the stronger question iii'), it suffices to show that for $\pi :\mathbb{C}^n\times \mathbb{C}^n\rightarrow \mathbb{C}^n$ the projection to the second factor, there is no point $y\in Y$ for which the intersection $\pi ^{-1}(y)\cap \Gamma$, here $\Gamma$ is the one constructed in the proof of Theorem \ref{Theorem2}, is empty. Note that the condition that $\pi ^{-1}(y)\cap \Gamma =\emptyset$ can be described in terms of polynomial equations by effective Hilbert Nullstellensatz. Hence, the maps from question iii) (call their parameter space, in terms of the coefficients of $F$ only, $W_1$) whose image are not the whole of $Y$ can be parametrised (in terms of the coefficients of $F$ only) by a constructible subset $W_2$ of $W_1$. By \cite{harris-michalek--sertoz}, $W_1$ and $W_2$ would be explicitly constructed from $X,Y$ and $d$. Then the parameter space for iii') is the constructible set $W=W_1\backslash W_2$. In particular, it is an algebraic variety.

We end this subsection by giving a proof for the fact that question iv') is also computable. By question iv), we can check whether $f$ is a regular embedding into $Y$. Denote by $Z=f(X)\subset Y$, which we do not know in advance. We need to have that $Z$ is a closed subvariety of $Y$. Let $\overline{Z}$ be the closure of $Z$ in $Y$, then $\overline{Z}$ is a closed subvariety of $Y$ of the same dimension as that of $X$. As in the proof of Theorem \ref{Theorem1}, the degree of $\overline{Z}$ is explicitly bounded. Since in question iv') we want to show that $f(X)$ is a closed subvariety of $Y$, it follows that $f(X)=\overline{Z}$, and hence $\overline{Z}$ must also be smooth. Then by Lemma \ref{Lemma2}, the ideal $I(Z)$ is generated by an explicitly bounded number of polynomials of degrees $\leq \deg(\overline{Z})$. (Again, these polynomials are undetermined, but we can parametrise them by their coefficients whose number is explicitly bounded.) We add to these polynomials also a finite set of generators for the variety $Y$ to make sure that $\overline{Z}$ is contained in $Y$. We also add in the equations for the condition that $\overline{Z}$ is smooth (which should be, if it is to be the image of $X$ under an isomorphism $f$). Denote by $\pi :\mathbb{C}^n\times \mathbb{C}^n\rightarrow \mathbb{C}^n$ the projection to the second factor.  Then as in the previous paragraph, the fact that $f(X)=\overline{Z}$ is the same as that there is no $z\in \overline{Z}$ for which $\pi ^{-1}(z)\cap \Gamma =\emptyset$. Hence, as before, among the regular embeddings $f$ of $X$ into $Y$, those for which $f(X)$ is not a closed subvariety is parametrised by a constructible subset of the variety parametrising the answer to question iv). Hence, the parameter space for iv'), which is the difference between these two varieties, is also a constructible set and hence an algebraic variety. 

\subsection{Maps on arbitrary algebraic varieties} Since any algebraic variety has a finite Zariski open cover by affine varieties, the results in the previous subsection imply that all questions i), ii), iii), iii'), iv) and iv') in the previous subsection are also computable when $X$ and $Y$ are arbitrarily irreducible algebraic varieties. To this end, it is enough to make precise for a rational map $f:X\dashrightarrow Y$, what it means to have that $f$ has bounded degree. To this end, we can proceed as follows.  Let $X_1,\ldots ,X_p$ be a Zariski open covering by affine varieties for $X$, and $Y_1,\ldots ,Y_q$ be a Zariski open covering by affine varieties for $Y$. Assume that $X_i$ and $Y_j$ belong to $\mathbb{C}^n$ for all $i,j$. Then we say that the degree of $f$ is bounded by a positive integer $d$ if for every pair $i,j$ for which $f$ maps $X_i$ into $Y_j$, there is a rational map $F_{i,j}:\mathbb{C}^n\dashrightarrow \mathbb{C}^n$ of degree $\leq d$ and so that $F_{i,j}|_{X_i}=f|_{X_i}$.

\section{Concluding remarks}

In this paper we showed that birational maps of bounded degrees between algebraic varieties can be parametrised by some algebraic varieties. We provide an explicit algorithm for implementing on computers. We also proved that similar results hold for other interesting classes of maps, such as biregular isomorphisms. These results are valid for both varieties over $\mathbb{C}$ and over fields of positive characteristic. They provide countable invariants for characterising these maps. Based on these results, together with Iitaka's fibrations, we proposed a rough approach towards solving the birationality problem, computationally and effectively, in general. 

An application of the above results is the following. Let $X$ be a smooth algebraic affine curve. Then whether or not $X$ is algebraically embedded in the affine plane $\mathbb{A}^2$ is characterised by a countable set of varieties explicitly constructed from $X$. In fact, $X$ is algebraically embedded into $\mathbb{A}^3$, and hence we can apply the results in Subsection \ref{SectionAffine}. The question of whether there is a finite set of invariants characterising that a smooth affine algebraic curve is algebraically embeddable into $\mathbb{C}^2$ is a long standing, classical open question.  

In the remarks after the proof of Theorem \ref{TheoremMain}, we proposed an approach toward the birationality problem. If it is true that $X$ and $Y$ are {\bf not} birationally equivalent, and we want to confirm this, then we can try the following alternative approach. That $X$ and $Y$ are not birationally equivalent is the same as that all the systems $W(X,Y,d)$ ($d=1,2,3,\ldots $) have no solutions, which by effective Hilbert's Nullstellensatz is the same as having an identity $1\equiv\sum \tau _ih_i$, where $h_i$ are a generator for the ideal of $W(X,Y,d)$, and all $\tau _i$ and $h_i$ have degrees bounded in terms of $d$ and $X,Y$. The systems $W(X,Y,d+1)$ contains as a special case the system $W(X,Y,d)$, for all $d$. Hence if we are able to verify the existence of such an identity for small values of $d$, we may be able to use induction to deduce the identity for general $d$. 

It is interesting to know if in Theorems \ref{Theorem2} and \ref{Theorem3} we can get rid of the smoothness assumptions on the variety $X$.

\end{document}